\documentclass[11pt]{dk_article}
\usepackage[top=3cm, bottom=3cm, left=3cm, right=3cm]{geometry}
\usepackage{graphicx}
\usepackage{todonotes}
\usepackage{color}
\usepackage{hyperref}
\newcommand{\dk}{}
\definecolor{darkred}{rgb}{.7,0,0}

\definecolor{darkgreen}{rgb}{0,0.7,0}

\definecolor{darkblue}{rgb}{0,0,0.7}

\usepackage[utf8]{inputenc}

\newcommand\xqed[1]{%
  \leavevmode\unskip\penalty9999 \hbox{}\nobreak\hfill
  \quad\hbox{#1}}
\newcommand\rmkend{\xqed{$\blacksquare$}}

\title{Well-Posedness And Accuracy Of The Ensemble Kalman Filter In\\ 
Discrete And Continuous Time}
\author{D.T.B. Kelly, K.J.H. Law, A.M. Stuart}
\institute{The University of Warwick, \email{David.Kelly@Warwick.ac.uk}}

\newcommand{\naturals}{\mathbb{N}}

\newcommand{\Gammatilde}{\tilde{\Gamma}}

\newcommand{\ebar}{\bar{e}}

\newcommand{\vhat}{\widehat{v}}
\newcommand{\uhat}{\widehat{u}}
\newcommand{\mhat}{\widehat{m}}

\newcommand{\argmin}{{\rm argmin}}

\newcommand{\vto}[1]{v^{(#1)}}

\newcommand{\filtration}{\mathcal{F}}
\newcommand{\Chat}{\widehat{C}}

\newcommand{\CK}{\mathcal{K}}
\newcommand{\BP}{\mathbf{P}}
\newcommand{\BE}{\mathbf{E}}
\newcommand{\T}{\mathbb{T}}
\newcommand{\C}{\mathbb{C}}
\newcommand{\R}{\mathbb{R}}
\newcommand{\pd}{\partial}
\newcommand{\cH}{\mathcal{H}}
\newcommand{\Z}{\mathbb{Z}}
\newcommand{\Pl}{\mathcal{P}_{\lambda}}
\newcommand{\Ql}{\mathcal{Q}_{\lambda}}

\newcommand{\eto}[1]{e^{(#1)}}

\newcommand{\normone}[1]{\left| #1 \right|}
\newcommand{\normtwo}[1]{\left\| #1 \right\|}

\newcommand{\vbar}{\bar{v}}

\newcommand{\bigoh}{O}

\newcommand{\trace}{{\rm tr}}
\newcommand{\E}{\mathbf{E}}
\renewcommand{\P}{\mathbf{P}}

\newcommand{\CV}{\mathcal{V}}
\newcommand{\CL}{\mathcal{L}}

\begin{document}
\maketitle

\begin{abstract}
The ensemble Kalman filter (EnKF) is a method for combining a dynamical model
with data in a sequential fashion. Despite its widespread use, there
has been  little analysis of its theoretical properties. Many of the
algorithmic innovations associated with the filter, which are required 
to make a useable algorithm in practice, are derived in an {\em ad hoc}
fashion. The aim of this paper is to initiate the development of a
systematic analysis of the EnKF, in particular to do so in the small 
ensemble size limit. The perspective is to view the method as a state estimator, 
and not as an algorithm which approximates the true filtering distribution.
The perturbed observation version of the algorithm is studied, without and
with variance inflation. Without variance inflation well-posedness of
the filter is established; with variance inflation 
accuracy of the filter, with resepct to the true signal underlying 
the data, is established. The algorithm is considered in discrete time, and 
also for a continuous time limit arising when observations are frequent and subject
to large noise. The underlying dynamical model,
and assumptions about it, is sufficiently general to include the Lorenz '63
and '96 models, together with the incompressible
Navier-Stokes equation on a two-dimensional
torus. The analysis is limited to the case of complete observation of the
signal with additive white noise.  {Numerical results are presented for the 
Navier-Stokes equation on a two-dimensional torus for both complete and 
partial observations of the signal with additive white noise.}
\end{abstract}

\section{Introduction}
In recent years the ensemble Kalman filter (EnKF) \cite{eve06} has become a widely
used methodology for combining dynamical models with data. 
The algorithm is used in oceanography, oil reservoir simulation
and weather prediction \cite{BEVL98,EVL00,kal03,orl08}, for example. 
The basic idea of the method is to propagate an ensemble of particles
to describe the distribution of the signal given data, employing
empirical second order statistics to update the distribution in a Kalman-like
fashion when new data is acquired.
Despite the widespread use of the method, its behaviour is not well understood. In contrast with the ordinary Kalman filter, which applies to linear
Gaussian problems, it is difficult to find a mathematical justification for EnKF. The most notable progress in this direction can be found in \cite{legland,mandel2011convergence}, 
where it is proved that, for linear dynamics, the EnKF approximates
the usual Kalman filter in the large ensemble limit. This analysis is however 
far from being useful for practitioners who typically run the method with
small ensemble size on nonlinear problems. 
Furthermore there is an accumulation of numerical evidence showing that the EnKF,
and related schemes such as the extended Kalman filter,
can ``diverge'' with the meaning of ``diverge'' ranging from simply loosing
the true signal through to blow-up \cite{IKJ02,MH08,gottwald2013}. 
The aim of our work is to try and build mathematical foundations for the
analysis of the EnKF, in particular with regards to well-posedness
(lack of blow-up) and accuracy (tracking the signal over arbitrarily long
time-intervals). To make progress on such questions it is necessary to impose
structure on the underlying dynamics and we choose to work with dissipative
quadratic systems with energy-conserving nonlinearity, a class of problems
which has wide applicability \cite{MW06} and which
has proved to be useful in the
development of filters \cite{majda2012filtering}. 

In section \ref{sec:enkf} we derive the perturbed observation form
of the EnKF and demonstrate how it links to the randomized maximum likelihood
method (RML) which is widely used in oil reservoir simulation \cite{orl08}.
We also introduce the idea of variance inflation, widely used in many
practical implementations of the EnKF \cite{anderson2007adaptive}. 
Section \ref{sec:enkfd} contains theoretical analyses of the perturbed
observation EnKF, without and with variance inflation. Without
variance inflation we are able only to prove bounds which grow exponentially
in the discrete time increment underlying the algorithm (Theorem \ref{thm:disc});
with variance inflation we are able to prove filter accuracy and show that, in
mean square with respect to the noise entering the algorithm, the filter is uniformly close to the true signal for all large
times, provided enough inflation is employed (Theorem \ref{thm:disc_varinf}). 
These results, and in particular the one concerning variance inflation,
are similar to the results developed in \cite{brett2012accuracy}
for the 3DVAR filter applied to the Navier-Stokes equation and
for the 3DVAR filter applied to the Lorenz '63 model in \cite{LSS13}, 
as well as the similar analysis developed in \cite{moodey2013nonlinear} for
the 3DVAR filter applied to globally Lipschitz nonlinear dynamical systems.
In section \ref{sec:cts} we  describe a continuous time limit in which
data arrives very frequently, but is subject to large noise. If these
two effects are balanced appropriately a stochastic (partial) differential equation
limit is found and it is instructive to study this limiting continuous time
process. This idea was introduced in \cite{DLSZ12} for the 3DVAR filter
and is here employed for the EnKF filter. The primary motivation for
the continuous time limit is to obtain insight into the mechanisms underlying
the EnKF; some of these mechanisms are more transparent in continuous time. 
In section \ref{sec:enkfc} we analyze the well-posedness of the continuous
time EnKF (Theorem \ref{thm:cts}).
Section \ref{sec:num} contains numerical experiments which illustrate
and extend the theory, and section \ref{sec:conc} contains some brief 
concluding remarks.

Throughout the sequel we use the following notation. 
Let $\CH$ be a separable Hilbert space with norm $\normone{\cdot}$ and inner product $\inner{\cdot,\cdot}$. For a linear operator $C$ on $\CH$, we will write
$C \geq 0$ (resp. $C>0$) when $C$ is self-adjoint and positive semi-definite (resp.
positive definite). Given $C>0$, we will denote
\begin{equ}
\normone{\cdot}_C \defin \normone{C^{-1/2}(\cdot)}\;.
\end{equ}

\section{Set-Up}\label{sec:enkf}

\subsection{Filtering Distribution}
We assume that the observed dynamics are governed by an evolution equation
\begin{equ}\label{e:evo}
\frac{du}{dt} = F(u)\;
\end{equ}
which generates a one-parameter semigroup $\Psi_t:\CH \to \CH$. We also assume that $\CK \subset \CH$ is another Hilbert space, which acts as the \emph{observation space}. We 
assume that noisy observations are made in $\CK$ every $h$ time units
and write $\Psi=\Psi_h.$ We define $u_{j} = u(jh)$ for $j\in\naturals$ and, assuming
that $u_0$ is uncertain and modelled as Gaussian distributed, we obtain
\begin{equ}
u_{j+1} = \Psi(u_j) \;, \quad\text{with $u_0 \sim N(m_0,C_0)$} \;
\end{equ} 
for some initial mean $m_0$ and covariance $C_0.$ We are given the observations
\begin{equ}
y_{j+1} = Hu_{j+1} + \Gamma^{1/2}\xi_{j+1}\;, \quad\text{with $\xi_j \sim N(0,I)$
i.i.d.}\;,
\end{equ}
where $H \in \CL(\CH,\CK)$ is the \emph{observation operator} and $\Gamma \in \CL(\CH,\CH)$ with $\Gamma \geq 0$ is the covariance operator of the observational noise; the
i.i.d. noise sequence $\{\xi_j\}$ is assumed independent of $u_0.$ The
aim of filtering is to approximate the distribution of $u_{j}$ given $Y_j=\{y_{\ell}\}_{\ell=1}^j$ using a sequential update algorithm. That is, given the distribution $u_j|Y_j$ as well as the observation $y_{j+1}$, find the distribution of $u_{j+1}|Y_{j+1}$.
We refer to the sequence $\P(u_j|Y_j)$ as the {\bf filtering distribution}.

\subsection{Assumptions}\label{sec:nse}
To write down the EnKF as we do in section \ref{sec:enkf}, and indeed to derive the continuum
limit of the EnKF, as we do in section \ref{sec:cts}, 
we need make no further assumptions about the underlying dynamics and observation
operator other than those made above. However,
in order to analyze the properties of the EnKF, as we do in sections \ref{sec:enkfd}
and \ref{sec:enkfc}, we will need
to make structural assumptions and we detail these here. 
It is worth noting that the assumptions we make on the underlying dynamics
are met by several natural models used to test data assimilation algorithms.
In particular, the 2D Navier-Stokes equations on a torus, as well as both Lorenz '63 and '96 
models, satisfy Assumptions \ref{a:1} \cite{MW06,majda2012filtering,Tem99}. 
\par

\begin{ass} ({\bf Dynamics Model})
\label{a:1}
Suppose there is some Banach space $\CV$, equipped with norm $\normtwo{\cdot}$, that can be continuously embedded into $\CH$. We assume that \eqref{e:evo} has the form 
\begin{equ}
\frac{du}{dt} + \CA u + \CB (u,u) = f\;, 
\label{e:nse}
\end{equ}
where $\CA : \CH \to \CH$  is an unbounded linear operator satisfying 
\begin{equ}\label{e:linear1}
\inner{\CA u,u} \geq \lambda \normtwo{u}^2\;,
\end{equ}
for some $\lambda > 0$, $\CB$ is a symmetric bilinear operator $\CB : \CV \times \CV \to \CH$ and $f : \reals_+ \to \CH$. We furthermore assume that $\CB$ satisfies the identity
\begin{equ}\label{e:bilinear1}
\inner{\CB(u,u),u} = 0\;,
\end{equ}
for all $u\in\CH$ and also
\begin{equ}\label{e:bilinear2}
\inner{\CB(u,v),v} \leq c \normtwo{u} \normtwo{v} \normone{v}\;,
\end{equ}
for all $u,v,w \in \CH$, where $c>0$ depends only on the bilinear form. 
We assume that the equation \eqref{e:nse} has a unique
weak solution for all $u(0) \in \CH$, and
generates a one-parameter semigroup $\Psi_t: \CV \to \CV$ which may be
extended to act on $\CH.$ Finally we assume that there exists
a global attractor $\Lambda \subset \CV$ for the
dynamics, and constant $R>0$ such that for any initial condition 
$u_0\in\Lambda$, we have that $\sup_{t\geq 0} \normtwo{u(t)} \le  R$.
\rmkend
\end{ass}

\begin{remark}
In the finite dimensional case the final assumption on the existence
of a global attractor does not need to be made as it is a consequence
of the preceding assumptions made. To see this note that
\begin{equ}
\frac12 \frac{d|u|^2}{dt} + \lambda\|u\|^2  \le  \langle f,u \rangle. 
\label{e:est}
\end{equ}
The continuous embedding of $\CV$, together with the Cauchy-Schwarz inequality,
implies the existence of a strictly positive constant $\epsilon$ 
such that
\begin{equ}
\frac12 \frac{d|u|^2}{dt} + \epsilon|u|^2  \le  \frac{1}{2\delta}|f|^2+\frac{\delta}{2}|u|^2 
\label{e:est}
\end{equ}
for all $\delta>0$. Choosing $\delta=\epsilon$ gives the existence of
an absorbing set and hence a global attractor by Theorem 1.1 in Chapter I
of \cite{Tem99}.
In infinite dimensions the existence of a global attractor in $\CV$ follows
from the techniques in \cite{Tem99} for the Navier-Stokes equation
by application of more subtle inequalities relating to the bilinear
operator $\CB$ -- see section 2.2 in Chapter III of \cite{Tem99}.  
Other equations arising in dissipative fluid mechanics can be treated similarly.
\rmkend
\end{remark}

Whilst the preceding assumptions on the underlying dynamics apply to
a range of interesting models arising in applications, the following
assumptions on the observation model are rather restrictive; however we
have been unable to extend the analysis without making them. We will demonstrate, by means of numerical experiments, that our results 
extend beyond the observation scenario employed in the theory
\begin{ass} ({\bf Observation Model})
\label{a:2}
The system is completely observed so that $\CK=\CH$ and $H=I.$ Furthermore the
i.i.d. noise sequence $\{\xi_j\}$ is white so that $\xi_1 \sim N(0,\Gamma)$ with
$\Gamma=\gamma^2 I.$
\rmkend
\end{ass}

The following consequence of Assumption \ref{a:1} will be useful to us.

\begin{lemma}
\label{l:wpnse}
Let Assumptions \ref{a:1} hold.
Then there is $\beta \in \reals$ such that, for any $v_0 \in \Lambda$, 
$h>0$ and $w_0 \in \CH$,
\begin{equ}
|\Psi_h(v_0)-\Psi_h(w_0)| \leq e^{\beta h} \normone{v_0-w_0}\;.
\end{equ}
\end{lemma}

\begin{proof} 
Let $v,w$ denote the solutions of \eqref{e:nse} with
initial conditions $v_0, w_0$ respectively; define $e=v-w.$
Then
\begin{equ}
\frac{de}{dt} + \CA e +2 \CB (v,e) - \CB (e,e) = 0\;, 
\label{e:nse_error}
\end{equ}
with $e(0)=v_0-w_0.$
Taking the inner-product with $e$, using \eqref{e:linear1},
\eqref{e:bilinear1} and \eqref{e:bilinear2}, and choosing
$\delta=\lambda/(2cR)$, gives
\begin{align*}
\frac12 \frac{d}{dt}|e|^2+\lambda\|e\|^2 & \le 2c\|v\|\|e\| |e|\\
& \le 2cR\|e\| |e|\\
& \le cR(\delta \|e\|^2+\delta^{-1}|e|^2)\\
& =\frac{\lambda}{2}\|e\|^2+\frac{2}{\lambda}(cR)^2|e|^2.
\end{align*}
Thus
$$\frac{d}{dt}|e|^2 \le \frac{4}{\lambda}(cR)^2|e|^2$$
and the desired result follows from an application of
the Gronwall inequality.
\end{proof}

\section{The Ensemble Kalman Filter}
\label{sec:enkf}
\subsection{The Algorithm}

The idea of the EnKF is to represent the filtering distribution 
through an ensemble of particles, to propagate this ensemble under the model to
approximate the mapping $\P(u_j|Y_j)$ to $\P(u_{j+1}|Y_j)$ (refered to
as {\bf prediction} in the applied literature),
and to update the ensemble distribution to include the data point $Y_{j+1}$ 
by using a Gaussian approximation based on the second order statistics
of the ensemble (refered to as {\bf analysis} in the applied literature).  

The prediction step is achieved by simply flowing forward the ensemble under the model dynamics, that is 
\begin{equ}
\vhat_{j+1}^{(k)} = \Psi(v_j^{(k)})\;, \quad \text{for $k=1\dots K$}.
\end{equ}
The analysis step is achieved by performing a randomised version of the Kalman update formula, and using the empirical covariance of the prediction ensemble to compute the Kalman gain. 
There are many variants on the basic EnKF idea
and we will study the perturbed observation form of the method.

\par  
The algorithm proceeds as follows. 
\begin{enumerate}
\item Set $j=0$ and draw an independent set of samples 
$\{v_0^{(k)}\}_{k=1}^K$ from $N(m_0,C_0)$. 
\item (Prediction) Let $\vhat_{j+1}^{(k)} = \Psi(v_{j}^{(k)})$ and define $\Chat_{j+1}$ as the empirical covariance of $\{\vhat_{j+1}^{(k)}\}_{k=1}^K$. 
That is, 
\begin{equ}
\Chat_{j+1} = \frac{1}{K}\sum_{k=1}^K (\vhat_{j+1}^{(k)} - \vbar_{j+1})\otimes(\vhat_{j+1}^{(k)} - \vbar_{j+1})\;,
\end{equ}
where $\vbar_{j+1} = \frac{1}{K} \sum_{k=1}^K \vhat_{j+1}$ denotes the ensemble mean.
\item(Observation) Make an observation $y_{j+1} = Hu_{j+1} + \Gamma^{1/2}\xi_{j+1}$. Then, for each $k=1\dots K$, generate an artificial observation
\begin{equ}
y_{j+1}^{(k)} = y_{j+1} + \Gamma^{1/2}\xi_{j+1}^{(k)}\;,
\end{equ}
where $\xi_{j+1}^{(k)}$ are $N(0,I)$ distributed and pairwise independent.
\item (Analysis) Let $v_{j+1}^{(k)}$ be the minimiser of the functional
\begin{equ}
J(v) = \frac{1}{2}| y_{j+1}^{(k)} - v|_{\Gamma}^2 + \frac{1}{2}|\vhat_{j+1}^{(k)} - v|_{\Chat_{j+1}}\;.
\end{equ}
\item Set $j \mapsto j+1$ and return to step $2$. 
\end{enumerate}
The name ``perturbed observation EnKF'' follows from the construction of
the artificial observations $y_{j+1}^{(k)}$ which are found by perturbing the
given observation with additional noise.
The sequence of minimisers $v_{j+1}$ can be written down explicitly by simply solving the quadratic minimization problem. This straightforward exercise yields the following result.
\begin{prop}
The sequence $\{v_j^{(k)}\}_{j \geq 0}$ is defined by the equation
\begin{equ}
(I + \Chat_{j+1}H^T \Gamma^{-1} H)v_{j+1}^{(k)} = \vhat_{j+1}^{(k)} + \Chat_{j+1} H^T \Gamma^{-1}y_{j+1}^{(k)}\;,
\end{equ}
for each $k=1,\dots,K$.
\end{prop}
Hence, collecting the ingredients from the preceding,
the defining equations of the EnKF are given by
\begin{subequations}
\label{e:discrete_def}
\begin{align}
(I + \Chat_{j+1}H^T \Gamma^{-1} H)v_{j+1}^{(k)} &= \Psi(v_{j+1}^{(k)}) + \Chat_{j+1} H^T \Gamma^{-1}y_{j+1}^{(k)}\\
y_{j+1}^{(k)} &= y_{j+1} + \Gamma^{1/2}\xi_{j+1}^{(k)}\\ 
\vbar_{j+1} &= \frac{1}{K} \sum_{k=1}^K \Psi(v_{j+1}^{(k)})\\
\Chat_{j+1} &= \frac{1}{K}\sum_{k=1}^K \Bigl(\Psi(v_{j+1}^{(k)}) - \vbar_{j+1}\Bigr)\otimes\Bigl(\Psi(v_{j+1}^{(k)}) - \vbar_{j+1}\Bigr)\;.
\end{align}
\end{subequations}
There are other representations of the EnKF that are more algorithmically convenient, 
but the formulae \eqref{e:discrete_def} are better suited to our analysis.

\subsection{Connection to Randomized Maximum Likelihood}  
The analysis step of EnKF can be understood in terms of the 
Randomised Maximum Likelihood (RML) method widely used in oil reservoir
history matching applications \cite{orl08}. We will now briefly describe this method. Suppose that we have a random variable $u$ and that $u \sim N(\mhat,\Chat)$. Moreover, let $G$ be some linear operator and suppose we observe 
\begin{equ}
y = Gu + \xi \quad \text{where $\xi \sim N(0,\Gamma)\;.$}
\end{equ} 
One can use Bayes' theorem to write down the conditional density $\BP (u | y)$. In practice however, it is often sufficient (or sometimes even better) to simply have a collection of \emph{samples} $\{u^{(k)}\}_{k=1}^K$ from the conditional distribution, rather than the density itself. RML is a method of taking samples from the prior $N(\mhat,\Chat)$ and turning them into samples from the posterior. This is achieved as follows, given $\uhat^{(k)} \sim N(\mhat,\Chat)$ (samples from the prior), define $u^{(k)}$ for each $k=1 \dots K$ by $u^{(k)} = \argmin_u J^{(k)}(u)$ where
\begin{equ}
J^{(k)}(u) = \frac{1}{2}|y-Gu+\Gamma^{1/2}\xi^{(k)}|_{\Gamma}^2 + \frac{1}{2}|u-\uhat^{(k)}|_{\Chat}^2\;, 
\end{equ} 
where $\xi^{(k)}\sim N(0,I)$ and independent of $\xi$. The $u^{(k)}$ are then draws from the posterior distribution of $u|y$ which is a Gaussian with mean $m$ and covariance $C$. Since one can explicitly write down $(m,C)$, it may be checked that the $u^{(k)}$ defined as above are independent random variables of the form $u^{(k)} = m + C^{1/2}\zeta^{(k)}$, where $\zeta^{(k)} \sim N(0,I)$ i.i.d. and are hence draws from the desired posterior, as
we know show.
\begin{prop}
Assume that ${\hat C}$ is invertible. Then, in the above notation, we have that $u^{(k)} = m + C^{1/2}\zeta^{(k)}$, where $\zeta^{(k)} \sim N(0,I)$ i.i.d and $(m,C)$ are defined by
\begin{equs}
C^{-1} &= \Chat^{-1} + G^* \Gamma^{-1} G  \label{e:posterior1} \\
C^{-1} m &= G^* \Gamma^{-1} y + \Chat^{-1} \mhat \label{e:posterior2}\;.
\end{equs}
In particular, $u^{(k)}$ is a sample from the posterior of $u | y$. 
\end{prop}
\begin{proof}
Firstly, note that $C$ is invertible since ${\hat C}$ is invertible. Secondly, it is well known that the pair $(m,C)$ defined by \eqref{e:posterior1}, \eqref{e:posterior2} do indeed define the mean and covariance of the posterior. This can be easily verified by matching coefficients in the expression for the negative log-density 
\begin{equ}
\frac{1}{2}|y-Gu|_{\Gamma}^2 + \frac{1}{2}|u-\mhat|_{\Chat}^2\;.
\end{equ}
Hence, it suffices to verify that $u^{(k)} = m + C^{1/2}\zeta^{(k)}$. Since $\uhat^{(k)} \sim N(\mhat,\Chat)$, we can write $\uhat^{(k)} = \mhat + \Chat \eta^{(k)}$, for $\eta^{(k)} \sim N(0,I)$ i.i.d. Moreover, by matching coefficients in $J^{(k)}$, we see that 
\begin{equs}
(G^* \Gamma^{-1} G + \Chat^{-1})u^{(k)} &= G^* \Gamma^{-1} (y + \Gamma^{1/2} \xi^{(k)}) + \Chat^{-1} \uhat^{(k)}\\
&= \bigg(G^* \Gamma^{-1} y  + \Chat^{-1} \mhat \bigg)+ \bigg(G^*\Gamma^{-1/2} \xi^{(k)} + \Chat^{-1/2} \eta^{(k)}\bigg) \;.
\end{equs}
Using \eqref{e:posterior1}, this can be rewritten as
\begin{equ}
C^{-1}u^{(k)} = \bigg(G^* \Gamma^{-1} y  + \Chat^{-1} \mhat \bigg)+ \bigg(G^*\Gamma^{-1/2} \xi^{(k)} + \Chat^{-1/2} \eta^{(k)}\bigg) \;.
\end{equ}
Now, by \eqref{e:posterior1} and \eqref{e:posterior2} we have that 
\begin{equ}
m = C \bigg(G^* \Gamma^{-1} y  + \Chat^{-1} \mhat \bigg)
\end{equ}
and moreover, we see that 
\begin{equs}
 \BE & \bigg(C (G^*\Gamma^{-1/2} \xi^{(k)} + \Chat^{-1/2} \eta^{(k)}) \otimes  C(G^*\Gamma^{-1/2} \xi^{(k)} + \Chat^{-1/2} \eta^{(k)})\bigg)\\
& = C \big(G^* \Gamma^{-1} G + \Chat^{-1} \big) C = C\;.
\end{equs}
This completes the proof. 
\end{proof}

\par
The analysis step of perturbed observation EnKF fits into the above inverse problem framework, since we are essentially trying to find the conditional distribution of $u_{j+1}|Y_j$ given the observation $y_{j+1}$. Suppose we are given $\{v_{j}^{(k)}\}$ and think of this as a sample from an approximation to the distribution of $u_j | Y_j$. Then the ensemble $\{\vhat_{j+1}^{(k)}\} = \{\Psi(v_{j+1}^{(k)})\}$ can be thought of as a sample from 
an approximation to the distribution of $u_{j+1}| Y_j$. Now, define $v_{j+1}^{(k)}$ using the RML method, minimising the functional 
\begin{equ}
J^{(k)}(v) = \frac{1}{2}|y_{j+1}-Hv+\xi_{j+1}^{(k)}|_{\Gamma}^2 + \frac{1}{2}|v-\vhat^{(k)}|_{\Chat_{j+1}}^2\;, 
\end{equ} 
where $\xi_{j+1}^{(k)}$ are i.i.d. $N(0,\Gamma)$ and where the covariance $\Chat_{j+1}$ is defined as the empirical covariance  
\begin{equ}
\Chat_{j+1} = \frac{1}{K}\sum_{k=1}^K (\vhat_{j+1}^{(k)} - \vbar_{j+1})\otimes(\vhat_{j+1}^{(k)} - \vbar_{j+1})\;. 
\end{equ}
This is precisely the EnKF update step described in the algorithm above. There are several reasons that this update step only produces \emph{approximate} samples from the filtering distribution. First of all, the distribution $u_{j+1} | Y_j$ is certainly not Gaussian in general, unless the dynamics are linear, hence the RML method becomes an approximation of samples. And secondly, since this distribution is not in general Gaussian, the choice of $\Chat_{j+1}$ is another approximation.

Although the approximations outlined are clearly quite naive, the decision to use the empirical distribution instead of say the push-forward of the covariance $u_j | Y_j$ gives a huge advantage to the EnKF in terms of computational efficiency. Moreover, by avoiding linearization, the prediction ensemble exhibits more of the nonlinear dynamical effects present in the underlying model that are present in, say, the extended Kalman filter
\cite{jaz70}. However the method as implemented is prone to failures of various
kinds and a commonly used way of over-coming one of these, namely collapse of the
particles onto a single trajectory, is to use variance inflation. We explain this
next.

%
%

\subsection{Variance Inflation}

The minimization step of the EnKF computes an update which is a compromise
between the model predictions and the data. This compromise is weighted
by the empirical covariance on the model and the fixed noise covariance on the
data. The model typically allows for unstable (chaotic) divergence of
trajectories, whilst the data tends to stabilize.
Variance inflation is a technique of adding stability to the algorithm by 
increasing the size of the model covariance in order to weight the data
more heavily. The form of variance inflation that we will study is found by 
shifting the forecast covariance $\Chat$ by some positive definite matrix. That is, one sets
\begin{equ}
\Chat_{j+1} \mapsto \Chat_{j+1} + A \;,  
\end{equ}
in (\ref{e:discrete_def}a).
Here $A:\CH \to \CH $ is a linear operator with $A > 0$. 
Equation (\ref{e:discrete_def}a) becomes
\begin{equ}
(I + (A+  \Chat_{j+1})H^T \Gammatilde^{-1} H)v_{j+1}^{(k)} = \vhat_{j+1}^{(k)} + (A + \Chat_{j+1}) H^T \Gammatilde^{-1}y_{j+1}^{(k)}
\end{equ}
This has the effect of weighting the data more than the model.
Furthermore, by adding a positive definite operator, one eliminates the null-space of $\Chat_{j+1}$ 
(which will always be present if the number of ensemble members is smaller than the
dimension of $\CH$) effectively preventing the ensemble from becoming degenerate. 
A natural choice is
$A = \alpha^2 I$ where $\alpha \in \reals$ and $I$ is the identity operator. 
In the sequel it will become clear that variance inflation has the effect of strengthening a contractive term in the algorithm, leading to filter accuracy if $\alpha$ is chosen
large enough.

\section{Discrete-Time Estimates}
\label{sec:enkfd}
In this section, we will derive long-time estimates for the discrete-time EnKF, under the
Assumptions \ref{a:1} and \ref{a:2} on the dynamics and observation models
respectively. We study the algorithm without and then with variance inflation.
The technique is to consider evolution of the error between the filter and
the true signal underlying the data. To this end we define
\begin{equation}
\label{eq:error}
e_j^{(k)} = v_j^{(k)} - u_j.
\end{equation}
Throughout this section we use $\E$ to denote expectation with respect to the independent
i.i.d. noise sequences $\{\xi_j\}$ and $\{\xi_j^{(k)}\}$ and initial
conditions $u_0$ and $v_0^{(k)}.$

\subsection{Well-Posedness Without Variance Inflation}

\begin{thm}\label{thm:disc}
Let Assumptions \ref{a:1} and \ref{a:2} hold and consider the
algorithm \eqref{e:discrete_def}. Then 
\begin{equ}
\E \normone{e_j^{(k)}}^2 \leq e^{2\beta h j} \E \normone{e_0^{(k)}}^2 + 2K\gamma^2\Bigl( \frac{e^{2\beta h j }-1}{e^{2\beta h  }-1}\Bigr)\,
\end{equ}
for any $j\geq 1$. 
\end{thm}
\begin{proof}
Firstly note that, under Assumption \ref{a:2}, 
the update rule (\ref{e:discrete_def}a) becomes
\begin{equ}
(I + \frac{1}{\gamma^2}\Chat_{j+1})v_{j+1}^{(k)}=\Psi(v_j^{(k)})+\frac{1}{\gamma^2}\Chat_{j+1}y_{j+1}^{(k)}\;.
\end{equ}
Secondly note that the underlying signal satisfies 
\begin{equ}
(I + \frac{1}{\gamma^2}\Chat_{j+1})u_{j+1}=\Psi(u_j)+\frac{1}{\gamma^2}\Chat_{j+1}\Psi(u_j)\;.
\end{equ}
Thus, subtracting from (\ref{e:discrete_def}a), we obtain 
\begin{equ}
(I + \frac{1}{\gamma^2}\Chat_{j+1})\eto{k}_{j+1} = \Psi(\vto{k}_j) - \Psi(u_j) +\frac{1}{\gamma^2}\Chat_{j+1}(y_{j+1}^{(k)} - \Psi(u_{j}))\;.
\end{equ}

 Now, if we define $r_1$ and $r_2$ by 
 \begin{align}
 (I + \frac{1}{\gamma^2}\Chat_{j+1})r_1 &=  \Psi(\vto{k}_{j} ) - \Psi(u_j)\\
 (I + \frac{1}{\gamma^2}\Chat_{j+1})r_2 &=\frac{1}{\gamma^2}\Chat_{j+1}(y_{j+1}^{(k)} - \Psi(u_{j}))\;,
 \end{align}
 then $\eto{k}_{j+1} = r_1 + r_2$. Moreover, since $\dk{\Chat_{j+1}}$ is symmetric and positive semi-definite, we have that 
\begin{equ}
\Bigl|\bigl(I + \frac{1}{\gamma^2}\Chat_{j+1}\bigr)^{-1}\Bigr|\leq1 \quad \text{and}\dk{\quad \Bigl|\bigl(I + \frac{1}{\gamma^2}\Chat_{j+1}\bigr)^{-1}\frac{1}{\gamma^2}\Chat_{j+1}\Bigr|\leq 1}\;.
\end{equ} 
Note also that $\Chat_{j+1}$ has rank $K$ and let $P_{j+1}$ denote
projection into the finite dimensional subspace orthogonal to
the kernel of $\Chat_{j+1}$. Then
\begin{align*}
\frac{1}{\gamma^2}\Chat_{j+1}(y_{j+1}^{(k)} - \Psi(u_{j}))
&=\frac{1}{\gamma^2}\Chat_{j+1}P_{j+1}(y_{j+1}^{(k)} - \Psi(u_{j}))\\
&=\frac{1}{\gamma^2}\Chat_{j+1}P_{j+1}(\xi_{j+1}+\xi_{j+1}^{(k)}).
\end{align*}
It follows from this and from Lemma \ref{l:wpnse} that
 \begin{equ}
 |r_1|\leq \normone{\Psi(\vto{k}_{j} ) - \Psi(u_j)} \leq e^{\beta h} \normone{\eto{k}_j}\;,
 \end{equ}
 and 
 \begin{equ}
 |r_2| \leq \normone{y_{j+1}^{(k)} - \Psi(u_{j})} = \normone{P_{j+1}(\xi_{j+1}+\xi_{j+1}^{(k)})}\;.
 \end{equ}
Now, if we let $\filtration_j$ be the $\sigma$-algebra generated by $\{\eto{k}_1,\dots,\eto{k}_j\}_{k=1}^K$ then, since $r_2$ has zero mean and is
conditionally independent of $r_1$, we have
\begin{equ}
 \E\left(\normone{\eto{k}_{j+1}}^2 | \filtration_j \right) = \normone{r_1}^2 + \E\bigl|P_{j+1}(\xi_{j+1}+\xi^{(k)}_{j+1})\bigl|^2 \leq e^{2\beta h} \normone{\eto{k}_j}^2 +2K\gamma^2\;. 
 \end{equ}
Here we have used the fact that $P_{j+1}$ projects onto a space of dimension
at most $K$. 
It follows that 
 \begin{equ}
 \E\normone{\eto{k}_{j+1}}^2 = \E \left( \E\left(\normone{\eto{k}_{j+1}}^2 | \filtration_j \right) \right) \leq  e^{2\beta h}\E\normone{\eto{k}_{j}}^2 + 2K\gamma^2\;,
 \end{equ}
and the result follows from the discrete Gronwall inequality. 
\end{proof}

The preceding result shows that the EnKF is well-posed and does
not blow-up  faster than exponentially. We now show that, with the
addition of variance inflation, a stronger result can be proved,
implying accuracy of the EnKF. 

\subsection{Accuracy With Variance Inflation}

We will focus on the variance inflation technique with $A = \alpha^2 I$. In this setting, again assuming $H=I$ and $\Gamma = \gamma^2 I$, the EnKF ensemble is governed by the following update equations. 
\begin{equ}\label{e:enkf_varinf}
(I + \frac{\alpha^2}{\gamma^2} I +  \frac{1}{\gamma^2}\Chat_{j+1})v_{j+1}^{(k)} = \Psi(v_{j}^{(k)}) + (\frac{\alpha^2}{\gamma^2}I + \frac{1}{\gamma^2}\Chat_{j+1}) y_{j+1}^{(k)}\;.
\end{equ}
We will now show that with variance inflation, one obtains much stronger long-time estimates than without it. In particular, provided the inflation parameter $\alpha$ is large enough, the ensemble stays within a bounded region of the truth, in a root-mean-square sense.  
\begin{thm}\label{thm:disc_varinf}
Let $\{v^{(k)}\}_{k=1}^K$ satisfy \eqref{e:enkf_varinf} and let $e_j^{(k)} = v_j^{(k)} - u_j$. Let $\theta = \frac{\gamma^2 }{\gamma^2 + \alpha^2}e^{2\beta h}$, then
\begin{equ}
\BE |e_j^{(k)}|^2 \leq \theta^j \BE |e_0^{(k)}|^2 + 2K\gamma^2\frac{1-\theta^{j}}{1-\theta}\;,
\end{equ}
for all $j\in \naturals$. In particular, if $\theta < 1$ then 
\begin{equ}
\lim_{j\to\infty} \BE |e^{(k)}_j|^2 \leq \frac{2K\gamma^2}{1-\theta}\;.
\end{equ}
\end{thm}
\begin{proof}
The proof is almost identical to the proof of Theorem \ref{thm:disc}. The only difference is that here we use the estimates
\begin{equ}
|(I + \frac{\alpha^2}{\gamma^2} I +  \frac{1}{\gamma^2}\Chat_{j+1})^{-1}| \leq \frac{\gamma^2}{\alpha^2 + \gamma^2} \quad \text{and} \quad |(I + \frac{\alpha^2}{\gamma^2} I +  \frac{1}{\gamma^2}\Chat_{j+1})^{-1}(\frac{\alpha^2}{\gamma^2} I +  \frac{1}{\gamma^2}\Chat_{j+1})| \leq 1\;.
\end{equ}
Proceeding exactly as above, we obtain
\begin{equ}
\BE |e^{(k)}_{j+1}|^2 \leq \frac{\gamma^2}{\alpha^2 + \gamma^2} e^{\beta h}\BE |e^{(k)}_j|^2 + 2K^2 \gamma^2  = \theta \BE |e^{(k)}_j|^2 + 2K^2 \gamma^2\;,
\end{equ}
and the result follows from the discrete Gronwall inequality. 
\end{proof}

\begin{rmk}
Choosing $\alpha$ large enough to ensure $\theta<1$ will result in filter boundedness; furthermore, if the observational noise standard deviation $\gamma$ is small then choosing $\alpha$ large enough results in filter accuracy.
\end{rmk}

\section{Derivation Of The Continuous Time Limit}
\label{sec:cts}
In this section we formally derive the continuous time scaling limits of the EnKF. The idea is to rearrange the update equation such that it resembles the \emph{discretization} of a stochastic ODE/PDE; we will simply refer to this as an SDE, be it in finite or infinite dimensions. We shall see that non-trivial limits only arise in situations where the noise is rescaled. 
\par
First, observe from \eqref{e:discrete_def} that
\begin{align*}
v_{j+1}^{(k)} - v_{j}^{(k)} &= \vhat_{j+1}^{(k)}- v_{j}^{(k)} - \Chat_{j+1}H^T \Gamma^{-1} H v_{j+1}^{(k)}+ \Chat_{j+1} H^T \Gamma^{-1}y_{j+1}^{(k)}\\
&=\Psi_h(v_{j}^{(k)})- v_{j}^{(k)} - \Chat_{j+1}H^T \Gamma^{-1} H v_{j+1}^{(k)}+ \Chat_{j+1} H^T \Gamma^{-1}y_{j+1}^{(k)}
\end{align*}
Now, if we attempt to take $h\to 0$, then the third and fourth terms on the right hand side above will lead to divergences when added up, since they are $\bigoh(1)$. This can be avoided by choosing an appropriate rescaling for the noise sources.  To this end, let $\Gamma = h^{-s}\Gamma$ for some $s > 0$, then we have 
\begin{equ}
v_{j+1}^{(k)} - v_{j}^{(k)} = \Psi_h(v_{j}^{(k)})- v_{j}^{(k)} - h^{s}\Chat_{j+1}H^T \Gamma_0^{-1} H v_{j+1}^{(k)}+ h^{s}\Chat_{j+1} H^T \Gamma_0^{-1}y_{j+1}^{(k)}\;.
\end{equ}
Now, if we define the primitive $z$ of $y$ by
\begin{equ}
z_{j+1}^{(k)} -z_{j}^{(k)} = h y_{j+1}^{(k)} 
\end{equ}
then we have coupled difference equations
\begin{subequations}
\begin{align}\label{e:disc_cts}
v_{j+1}^{(k)} - v_{j}^{(k)} &= \Psi_h(v_{j}^{(k)})- v_{j}^{(k)} - h^{s}\Chat_{j+1}H^T \Gamma_0^{-1} H v_{j+1}^{(k)}+ h^{s-1}\Chat_{j+1} H^T \Gamma_0^{-1}(z_{j+1}^{(k)} -z_{j}^{(k)}) \\
z_{j+1}^{(k)} -z_{j}^{(k)} &= h H u_{j+1} + h^{1-s/2}\Gamma_0^{1/2} (\xi_{j+1}^{(k)} + \xi_{j+1})\;.
\end{align}
\end{subequations}
The final step is to find an SDE for which the above represents a reasonable numerical scheme. Of course, this depends crucially on the choice of scaling parameter $s$. In fact, 
it is not hard to show that the one non-trivial limiting SDEs corresponds to the choice $s=1$. To see why this is the only valid scaling, notice that \eqref{e:disc_cts} implies that $s\geq 1$, since otherwise the $\bigoh(h^s)$ terms would diverge when added up. Likewise, from the second equation we must have $1-s/2 \geq 1/2$, for otherwise the stochastic terms would diverge when summed up, in accordance with the central limit theorem. Hence we must choose $s=1$. 
\par
If we invoke the approximation
$$\Psi_h(v)-v \approx hF(v)$$ 
then, in the case $s=1$, the system \eqref{e:disc_cts} is a mixed implicit-explicit 
Euler-Maruyama type scheme for the SDE   
\begin{subequations}
\label{e:enkf_cts11}
\begin{align}
d v^{(k)}  &= F(v^{(k)})dt - C(v)H^T \Gamma_0^{-1} H v^{(k)}dt + C(v)H^T\Gamma_0^{-1}{dz^{(k)}} \\
dz^{(k)} &= Hudt + \Gamma_0^{1/2}({dW^{(k)}} + dB)\;.
\end{align}
\end{subequations}
Here $W^{(1)},\dots W^{(K)},B$ are pairwise independent cylindrical Wiener processes, arising as limiting processes of the discrete increments $\xi^{(1)},\dots \xi^{(K)}, \xi$ respectively. We use $v$ to denote the collection $\{v^{(k)}\}_{k=1}^K$
and the operator $C(v)$ is the empirical covaraince of the particles defined as
follows: 
\begin{subequations}
\label{e:discrete_def22}
\begin{align}
\vbar &= \frac{1}{K} \sum_{k=1}^K \Psi(v^{(k)})\\
C(v) &= \frac{1}{K}\sum_{k=1}^K \Bigl(v^{(k)} - \vbar\Bigr)\otimes\Bigl(v^{(k)} - \vbar\Bigr)\;.
\end{align}
\end{subequations}
Thus we have the system of SDEs \eqref{e:enkf_cts11} for $k=1, \dots, K$, coupled
together through \eqref{e:discrete_def22}.

\begin{remark}

If we substitute the expression for $dz^{(k)}$ from (\ref{e:enkf_cts11}b)
into (\ref{e:enkf_cts11}a) then we obtain 
\begin{equ}\label{e:spde_s1}
d v^{(k)} = F(v^{(k)})dt - C(v)H^T \Gamma_0^{-1} H (v^{(k)} - u)dt + C(v)H^T\Gamma_0^{-1/2}({dW^{(k)}} + {dB})\;.
\end{equ}
Of course in practice the truth $u$ is not known to us, but the equation \eqref{e:spde_s1} has a very clear structure which highlights the mechanisms at play
in the EnKF. The equation is given by the original dynamics with the
addition of two terms, one which pulls the solution of each ensemble
member back towards the true signal $u$, and a second which
drives each ensemble member with a sum of two white noises, one independently
chosen for each ensemble member (coming from the perturbed observations)
and the second a common noise (coming from the noise in the data).

The stabilizing term, which draws the ensemble member back towards the
truth, and the noise, both act only orthogonal to the null-space of
the empirical covariance of the set of particles.
The perturbed observations noise contribution
will act to prevent the particles from synchronizing
which, in their absence, could happen. If the particles
were to synchronize then the covariance disappears and we simply
obtain the original dynamics 
\begin{equ}
\label{eq:nw}
d v^{(k)}  = F(v^{(k)})dt \; 
\end{equ}
for each ensemble member. 

In this context it is worth noting that 
another approach to the derivation of a
continuous time limit is to never introduce 
the process $z$ and only think of the equation for $v^{(k)}$. In particular, we have
\begin{align*}
v_{j+1}^{(k)} - v_{j}^{(k)} = \Psi_h(v_{j}^{(k)})- v_{j}^{(k)} &- h^{s}\Chat_{j+1}H^T \Gamma_0^{-1} H v_{j+1}^{(k)}\\ &+ h^{s}\Chat_{j+1} H^T \Gamma_0^{-1}  \left(H u_{j+1} + h^{-s/2}\Gamma_0^{1/2}(\xi_{j+1}^{(k)} +  \xi_{j+1}) \right)\;.
\end{align*}
In this case, we still must have $s\geq 1$ in order to get a limit, 
but there is no requirement for $s\leq 1$. However, it is easy to see that in the case $s > 1$, one obtains the trivial scaling limit \eqref{eq:nw}
so that each ensemble member evolves according to the model dynamics and the
data is not seen. Such scalings are of no interest since they do not elucidate
the structure of the model/data trade-off which is the heart of the EnKF.. 
\rmkend
\end{remark}

\subsection{Limits With Variance Inflation}
With variance inflation, the update equation \eqref{e:disc_cts} becomes
\begin{equs}
v_{j+1}^{(k)} - v_{j}^{(k)} &= \Psi_h(v_{j}^{(k)})- v_{j}^{(k)} - h^{s}(A+\Chat_{j+1})H^T \Gamma_0^{-1} H v_{j+1}^{(k)}+ h^{s-1}(A+\Chat_{j+1}) H^T \Gamma_0^{-1}(z_{j+1}^{(k)} -z_{j}^{(k)}) \\
z_{j+1}^{(k)} -z_{j}^{(k)} &= h H u_{j+1} + h^{1-s/2}\Gamma_0^{1/2} (\xi_{j+1}^{(k)} + \xi_{j+1})\;. \label{e:disc_cts_varinf}
\end{equs}
By the same reasoning, it is clear that the only non-trivial
continuous time limit is given by
\begin{equs}\label{e:enkf_cts2}
d v^{(k)}  &= F(v^{(k)})dt - (A+C(v))H^T \Gamma_0^{-1} H v^{(k)}dt + (A+C(v))H^T\Gamma_0^{-1}{dz^{(k)}} \\
dz^{(k)} &= Hudt + \Gamma_0^{1/2}({dW^{(k)}} + dB)\;.
\end{equs}

\section{Continuous-Time Estimates}
\label{sec:enkfc}
In this section we obtain long-time estimates for the continuous time EnKF, under 
Assumptions \ref{a:1} and \ref{a:2}. These are the same
assumptions used in the discrete case and our rersults are
analogous to Theorems \ref{thm:disc}. 
Under Assumptions \ref{a:1} and \ref{a:2}, the continuous time EnKF equations 
\eqref{e:spde_s1} for the ensemble $v=\{v^{(k)}\}_{k=1}^K$ become 
\begin{equ}
\label{eq:need}
dv^{(k)}(t) + \left(\CA v^{(k)}(t) +\CB(v^{(k)}(t),v^{(k)}(t))\right)dt = f -\frac{1}{\gamma^{2}}C(v)(v^{(k)}(t)-u(t))dt + \frac{1}{\gamma}C(v)(dW^{(k)}(t) + dB(t))\;.
\end{equ}
We set $e^{(k)} = \vto{k}-u$, and write 
$e$ for $e=\{e^{(k)}\}_{k=1}^K$. Note that $C(v)=C(e)$ since shifting the
origin does not change the empirical covariance. Thus we have
\begin{equ}
d\eto{k}(t) + \left(\CA \eto{k} +\CB(\eto{k},\eto{k})+2\CB(\eto{k},u) \right)dt = -\frac{1}{\gamma^2}C(e)\eto{k} dt + \frac{1}{\gamma}C(e)(dW^{(k)}(t) + dB(t))\;.
\end{equ}
\begin{rmk}\label{rmk:soln}
In the next theorem, we will analyse the growth properties of solutions to the SPDE \eqref{eq:need}, but to make the statement precise we must specify what we mean by a solution. We use the standard notion of a strong solution as found in \cite{prato92}. In essence, we assume that the solution is strong enough so that all the terms in the integral expression
\begin{equs}\label{e:integral}
v^{(k)}(t) &+  \int_0^t \left(\CA v^{(k)}(s) +\CB(v^{(k)}(s),v^{(k)}(s))\right)ds\\ &= v^{(k)}(0) + \int_0^t \bigg( f -\frac{1}{\gamma^{2}}C(v(s))(v^{(k)}(s)-u(s))\bigg)ds + \frac{1}{\gamma}\int_0^t C(v(s))(dW^{(k)}(s) + dB(s))
\end{equs} 
do indeed make sense and moreover, fall into the domain of It\^o's formula. To be precise, we say that $v = \{v^{(k)}\}_{k=1}^K$ is a strong solution to \eqref{eq:need} over the interval $[0,T]$ if $v$ satisfies \eqref{e:integral} ($\BP$-a.s.) and moreover we have that
\begin{equ}
\int_0^T \normone{\CA v^{(k)}(t)} +\normone{\CB(v^{(k)}(t),v^{(k)}(t))} + \normone{C(v)(v^{(k)}(t)-u(t))} dt < \infty \quad \text{$\BP$-a.s.} 
\end{equ}
and
\begin{equ}
 \int_0^T \normtwo{C(v(t))}_{HS}^2 dt < \infty  \quad \text{$\BP$-a.s.}
\end{equ}
As can be seen in \cite[Theorem 4.17]{prato92}, these conditions are sufficient to utilise It\^o's formula. We note that, in the case of \eqref{eq:need}, it is not unreasonable to assume the existence of strong solutions. Indeed, in finite dimensions any type of solution will be a strong solution and moreover, the very existence of a Lyapunov function, as obtained in the theorem below, is enough to guarantee global solutions \cite{Mao97}. In infinite dimensions this is not the case in general. However the 
fact that the noise is effectively finite dimensional, due to the presence 
of the finite rank covariance operator C, does mean that existence
of strong solutions may well be established on a case-by-case basis
in some infinite dimensional settings. However it is difficult to do
this at the level of generality we study in this paper and hence we will 
not make any concrete statements concerning the existence of strong 
solutions, rather we will simply assume that one exists and is unique.
\rmkend
\end{rmk}
We may now state and prove the well-posedness estimate for this equation,
analogous to Theorem \ref{thm:disc}. We assume that \eqref{eq:need} has
a unique solution for all $t \ge 0$. In finite dimensions this
is in fact a consequence of the mean square estimate provided by the theorem;
since we have been unable to prove this in the rather general infinite dimensional
setting, however, we make it an assumption. We let $(\Omega,{\cal F},\P)$
denote the probability space underlying the independent initial conditions
and the driving Brownian motions $(\{W^{(k)}\}, B)$, and $\E$ denotes
expectation with respect to this space. 
\begin{thm}\label{thm:cts}
Assume that \eqref{eq:need} has a unique strong solution, in the sense of Remark \ref{rmk:soln} and that Assumptions \ref{a:1} and \ref{a:2}
hold. We then have that 
\begin{equ}
\frac{1}{K}\sum_{k=1}^K\E \normone{\eto{k}(t)}^2 \leq \left(\frac{1}{K}\sum_{k=1}^K\E \normone{\eto{k}(0)}^2\right) \exp\left({\frac{4(cR)^2t}{\lambda}}\right)\;,
\end{equ}
where $c>0$ is the constant appearing in \eqref{e:bilinear2}. Moreover, we have that
\begin{equ}
\frac{1}{K}\sum_{k=1}^K \int_0^t \E\normtwo{\eto{k}(s)}^2 ds \leq \left(\frac{1}{K}\sum_{k=1}^K\E \normone{\eto{k}(0)}^2 \right) \frac{1}{\lambda} \exp\left({\frac{4(cR)^2t}{\lambda}}\right)\;.
\end{equ}

\end{thm}
\begin{proof}
Using It\^o's formula, one can show that
\begin{align}
\E \normone{\eto{k}(t)}^2 = \E \normone{\eto{k}(0)}^2+\int_0^t &-2\E\inner{\eto{k}, \CA \eto{k}+ \CB(\eto{k},\eto{k})+2\CB(\eto{k},u)}ds\notag\\&+ \int_0^t \frac{-2}{\gamma^2}\E\inner{\eto{k},C(e)\eto{k}} + \frac{2}{\gamma^2}\E\trace\left(C(e)^2 \right)  ds\;.\label{e:cts_ito1}
\end{align}
Now, if we let $\{\xi_i\}_{i\in I}$ be some orthonormal basis of $\CH$, then we can simplify the above using the identity
\begin{align*}
\trace\left(C(e)^2\right) = \sum_{i\in I} \inner{C(e)\xi_i,C(e)\xi_i}\;.
\end{align*}
By expanding the right $C(e)$, we obtain
\begin{align*}
\trace\left(C(e)^2\right)&=
\sum_{i\in I}\frac{1}{K}\sum_{m=1}^K \inner{\eto{m}-\ebar, \xi_i}\inner{\eto{m},C(e)\xi_i}\\
&= \frac{1}{K}\sum_{m=1}^K \sum_{i\in I} \inner{\eto{m} -\ebar, \xi_i}\inner{C(e)\eto{m},\xi_i}\\
&= \frac{1}{K}\sum_{m=1}^K  \inner{\eto{m} -\ebar,C(e)\eto{m}}\\ &=  \frac{1}{K}\sum_{m=1}^K  \inner{\eto{m},C(e)\eto{m}} - \inner{\ebar,C(e)\ebar}\;. 
\end{align*}
Substituting this into \eqref{e:cts_ito1} and summing over $k=1 \dots K$, we obtain
\begin{align*}
\frac{1}{K}\sum_{k=1}^K\E \normone{\eto{k}(t)}^2 = \frac{1}{K}\sum_{k=1}^K\E \normone{\eto{k}(0)}^2&+\frac{1}{K}\sum_{k=1}^K\int_0^t -2\E\inner{\eto{k}, \CA \eto{k}+ \CB(\eto{k},\eto{k})+2\CB(\eto{k},u)}ds\\
&  - \frac{2}{\gamma^2}\int_0^t\E\inner{\ebar,C(e)\ebar} ds\;.
\end{align*}
And since $C(e)$ is positive semi-definite, we have
\begin{equ}
\frac{1}{K}\sum_{k=1}^K\E \normone{\eto{k}(t)}^2 \leq \frac{1}{K}\sum_{k=1}^K\E \normone{\eto{k}(0)}^2+\frac{1}{K}\sum_{k=1}^K\int_0^t -2\E\inner{\eto{k}, \CA \eto{k}+ \CB(\eto{k},\eto{k})+2\CB(\eto{k},u)}ds\;.
\end{equ}
Finally, using the assumptions on $\CA, \CB$, we have that 
\begin{align*}
-\inner{\eto{k}, \CA \eto{k}+ \CB(\eto{k},\eto{k})+2\CB(\eto{k},u)} &\leq -\lambda \normtwo{\eto{k}}^2 + 2\Bigl|\inner{\CB(\eto{k},u),\eto{k}}\Bigr|\\
&\leq -\lambda \normtwo{\eto{k}} +2c\normone{\eto{k}}\normtwo{\eto{k}}\normtwo{u} \\
&\leq -\lambda\normtwo{\eto{k}} + cR \left( \delta^{-1}\normone{\eto{k}}^2 + \delta \normtwo{\eto{k}}^2 \right)\;,
\end{align*}
recalling that $\sup_{t\geq 0}\normtwo{u(t)} \le R$. Putting this altogether, we have that
 \begin{align*}
 \frac{1}{K}\sum_{k=1}^K\E \normone{\eto{k}(t)}^2 +\frac{1}{K}\sum_{k=1}^K \int_0^t 2\left(\lambda - cR\delta  \right)\E\normtwo{\eto{k}(s)}^2 \leq  \frac{1}{K}\sum_{k=1}^K\E \normone{\eto{k}(0)}^2 + \frac{1}{K}\sum_{k=1}^K \int_0^t 2cR \delta^{-1} \E \normone{\eto{k}(s)}^2 ds \;.
 \end{align*}
 If we pick $\delta = \lambda/(2cR)$ then we obtain the estimate
\begin{equ}
\frac{1}{K}\sum_{k=1}^K\E \normone{\eto{k}(t)}^2 + \frac{1}{K}\sum_{k=1}^K \int_0^t \lambda\E\normtwo{\eto{k}(s)}^2 \leq  \frac{1}{K}\sum_{k=1}^K\E \normone{\eto{k}(0)}^2+ \frac{1}{K}\sum_{k=1}^K \int_0^t \frac{4 (cR)^2}{\lambda} \E \normone{\eto{k}(s)}^2 ds\;,
\end{equ}
and the result follows from Gronwall's inequality. As a consequence of this, we see that
\begin{align*}
\frac{1}{K}\sum_{k=1}^K \int_0^t \lambda\E\normtwo{\eto{k}(s)}^2 ds &\leq  \frac{1}{K}\sum_{k=1}^K\E \normone{\eto{k}(0)}^2+ \frac{1}{K}\sum_{k=1}^K \int_0^t \frac{4 (cR)^2}{\lambda} \E \normone{\eto{k}(s)}^2 ds \\
&\leq  \frac{1}{K}\sum_{k=1}^K\E \normone{\eto{k}(0)}^2+  \frac{4 (cR)^2}{\lambda}\left(\frac{1}{K}\sum_{k=1}^K\E \normone{\eto{k}(0)}^2\right)\int_0^t  \exp\left({\frac{4(cR)^2s}{\lambda}}\right)  ds \\
&= \frac{1}{K}\sum_{k=1}^K\E \normone{\eto{k}(0)}^2 \exp\left({\frac{4(cR)^2t}{\lambda}}\right) \;,
\end{align*} 
 which proves the second result and hence the theorem. 
\end{proof}
\begin{remark} 
%
%
In this case of non-trivial $H$ and $\Gamma$, the above argument does not work, but nevertheless it is still informative to see \emph{why} it doesn't work. Indeed, if we apply the exact same argument to the case of arbitrary $H,\Gamma$, we still obtain the identity
\begin{align*}
\frac{1}{K}\sum_{k=1}^K\E \normone{\eto{k}(t)}^2 = \frac{1}{K}\sum_{k=1}^K\E \normone{\eto{k}(0)}^2&+\frac{1}{K}\sum_{k=1}^K\int_0^t -2\E\inner{\eto{k}, \CA \eto{k}+ \CB(\eto{k},\eto{k})+2\CB(\eto{k},u)}ds\\
&  - 2\int_0^t\E\inner{\ebar,C(e)H^T \Gamma^{-1}H \ebar} ds\;.
\end{align*}
The reason we cannot proceed further is that even though $C(e)$ and $H\Gamma^{-1}H$ are themselves positive semi-definite and self adjoint, the same is not necessarily true for the product. 
\rmkend
\end{remark}
%


\section{Numerical Results}
\label{sec:num}

In this section we confirm the validity of the theorems derived 
in the previous sections for variants of the EnKF when applied
to the dynamical system \eqref{e:nse}.  Furthermore, we extend 
our numerical explorations beyond the strict range of validity 
of the theory and, in particular, consider the case of partial 
observations.  We conduct all of our numerical experiments in 
the case of the incompressible Navier-Stokes equation on a two dimensional torus.

We observe not only well-posedness, but indeed {\em boundedness} 
of the ensemble for both complete and partial observations, over 
long time-scales compared with the natural variability of the
dynamical system itself. However, the filter is always inaccurate 
when used without inflation.  We thus turn to study the effect
of inflation and note that our results indicate the filter 
can then always be made accurate, even in the case of partial observations,
provided that sufficiently many low Fourier modes are observed.  
In the case that only the high Fourier modes are observed the filter cannot be made accurate with inflation.

\subsection{Setup}
\label{sec:numset}

Let $\T^{2}$ denote the two-dimensional
torus of side $L:$ $[0,L) \times [0,L)$
with periodic boundary conditions. We
consider the equations
\begin{equation*}
\begin{array}{ccc}
\pd_{t}u(x, t) - \nu \Delta u(x, t)
 + u(x, t) \cdot \nabla u(x, t) + \nabla p(x, t) 
&=& f(x)  
\\
\nabla \cdot u(x, t) &=& 0 
\\
u(x, 0) &=& u_{0}(x) 
\end{array}
\end{equation*}
for all $x \in \T^{2}$ and $t\in(0, \infty)$. 
Here $u \colon \T^{2} \times (0, \infty) \to \R^{2}$ is a time-dependent vector field representing the velocity, 
$p \colon \T^{2} \times (0,\infty) \to \R$ is a time-dependent scalar field representing the pressure and $f \colon \T^{2} \to \R^{2}$ 
is a vector field representing the forcing which we take as 
time-independent for simplicity. The parameter $\nu$ 
represents the viscosity. We assume throughout that $u_0$
and $f$ have average zero over $\T^2$; it then follows
that $u(\cdot,t)$ has average zero over $\T^2$ for all
$t>0$. 

Define
$${\mathsf T}:= \left\{{\rm trigonometric}\,{\rm polynomials}\,\,
u:\T^2 \to {\mathbb R}^2\,\Bigl|\, \nabla \cdot u = 0, \,\int_{\T^{2}} u(x) \, dx = 0 \right\}
$$
and $\cH$ as the closure of ${\mathsf T}$ with respect to the
norm in $(L^{2}(\T^{2}))^{2} = L^{2}(\T^{2},\R^2)$.
We let $P:(L^{2}(\T^{2}))^{2} 
\to \cH$ denote the Leray-Helmholtz orthogonal projector. 
Given $m= (m_{1}, m_{2})^{\mathrm{T}}$, define $m^{\perp} := (m_{2}, -m_{1})^{\mathrm{T}}$. Then an orthonormal basis for (a complexified)
$\cH$ is given by $\psi_{m} \colon \T^{2} \to \C^{2}$, where 
\begin{equation}
\label{eq:fb}
\psi_{m} (x) := \frac{m^{\perp}}{|m|} \exp\Bigl(\frac{2 \pi i m \cdot x}{L}\Bigr)
\end{equation}
for $m \in \Z^{2} \setminus \{0\}$. 
Thus for $u \in \cH$ we may write
$$
u(x) = \sum_{m \in \Z^{2} \setminus \{0\}} u_{m} \psi_{m}(x)
$$
where, since $u$ is a real-valued function, we have the 
reality constraint $u_{-m} = - \overline{u_{m}}.$
We define the projection operators $\Pl: \cH \to \cH$ 
and $\Ql:\cH \to \cH$ for 
$\lambda \in\mathbb{N}\cup\{\infty\}$ 
by
$$(\Pl u\bigr)(x)  = \sum_{|2\pi m|^2 <\lambda L^2} u_{m} \psi_{m}(x),
\quad \Ql=I-\Pl.$$
Below we will choose the observation operator $H$ to be $\Pl$ or $\Ql$.
%

We define 
$A = -\nu P \Delta$
the Stokes operator, and, for every $s \in \R$, 
define the Hilbert spaces $\cH^s$ to be the domain of $A^{s/2}.$ 
We note that $A$ is diagonalized in $\cH$ in the basis comprised
of the $\{\psi_m\}_{m \in {\Z}^2\backslash\{0\}}$. 
We denote by $|\cdot|$ the norm on $\cH:=\cH^0$.  

Applying the projection $P$ to the Navier-Stokes
equation for $f=Pf$ we may write it as an ODE in $\cH$ as in \eqref{e:nse},
with $\CB(u,v)$ the {\it symmetric} 
bilinear form defined by 
$$\CB(u,v) = \frac12 P(u \cdot \nabla v)+ \frac12 P(v \cdot \nabla u)$$
for all $u,v\in\CV$.
See \cite{constantin1988navier} 
for details of this formulation of the Navier-Stokes
equation as an ODE in $\cH$.

We fix the domain size $L=2$. The forcing in $f$ is taken to be $f\propto\nabla^{\perp}\Psi$,
where $\Psi=\cos(\pi k_f \cdot x)$ and $\nabla^{\perp}=J\nabla$ with $J$
the canonical skew-symmetric matrix.
The parameters  
$(\nu,k_f,|f|)$ in (\ref{e:nse}), 
where $k_f$ is the wavevector of the forcing frequency,
are fixed throughout all of the experiments
shown in this paper at values which yield a chaotic regime;
specifically we take 
$(\nu,k_f,|f|) \approx \{0.01,(5,5),10\}$. 

Our first step in constructing a numerical experiment 
is to compute the true solution $u$ solving  equation \eqref{e:nse}. 
The true initial condition $u_0$ is randomly drawn from $N(0,\nu^2 A^{-2})$.
For all the experiments presented below, we will then begin with an initial
ensemble which is far from the truth, in order to probe the accuracy
and stability of the filter for the given 
parameters.  Specifically we let $m_0 \sim N(u(0),\beta \frac{4\pi^2 \nu}{L^2} A^{-1})$
and $v_0^{(k)} \sim N(m_0,(\beta/25)\frac{4\pi^2 \nu}{L^2}{A}^{-1})$ with 
$\beta=0.25$. 
Throughout, $\gamma=0.01$. We use the notation $m(t)$ to denote the mean
of the ensemble.

The method used to approximate the forward model
is a modification of a fourth-order Runge-Kutta method,
ETD4RK \cite{cox2002exponential}, in which the Stokes semi-group is
computed exactly by working in the incompressible Fourier basis
$\{\psi_{m}(x)\}_{m \in {\mathbb Z}^2\backslash\{0\}}$, and Duhamel's
principle (variation of constants formula) is used to incorporate the
nonlinear term.  We use a time-step of $dt = 0.005$.
Spatially, a Galerkin spectral method
\cite{hesthaven2007spectral} is used, in the same basis,
and the convolutions arising from products in the nonlinear term are
computed via FFTs.

Before proceeding with the numerical experiments concerning the
EnKF, it is instructive to run an experiment
in which $H=0$ so that the ensemble evolves according to the underlying attractor with
no observations taken into account.  This is shown in Figure \ref{fig:dfni1}.
Notice that the statistics of the ensemble remain 
well-behaved.  This is in stark contrast to the case of evolving ExKF without observations,
in which case the covariance will have an exponential growth rate 
corresponding asymptotically to the Lyapunov exponents of the attractor.
Figure \ref{fig:dfni1} sets a reference scale against which subsequent
experiments, which include observations, should be compared.

\begin{figure*} 
\center
\includegraphics[width=1\textwidth]{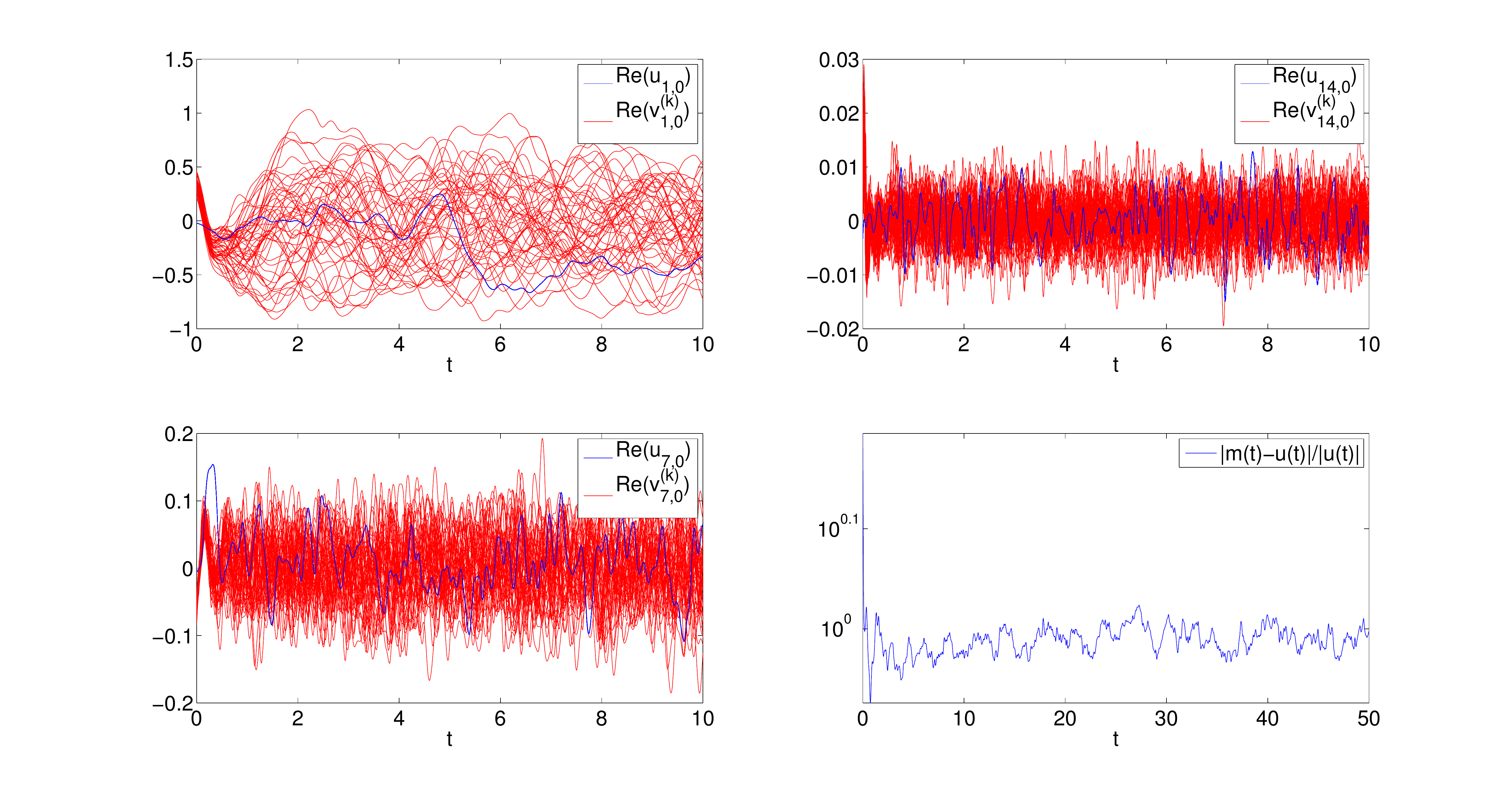}
\caption{Trajectories of various modes  (at observation times) 
of the estimators $v^{(k)}$ 
and the signal $u$ are depicted above 
along with the relative error
in the $L^2$ norm, $|m-u|/|u|$, 
for $H=\Pl$, with $\lambda=0$ -- i.e. nothing is observed.}
\label{fig:dfni1}
\end{figure*}

\subsection{Discrete Time}
\label{sec:numdisc}

Here we explore the case of discrete-time observations  
by means of numerical experiments, illustrating 
the results of section \ref{sec:enkfd}. We consider the cases $H=\Pl$ with
$\lambda=\infty$ so
that all Fourier modes represented on the grid are observed, as well as
both $H=\Pl$ and $H=\Ql$ with $\lambda<\infty$.

\subsubsection{Full observations}
\label{sec:discfull}

Here we consider observations made at all numerically 
resolved, and hence observable, wavenumbers in the system;
hence $K=32^2$, not including padding in the spectral domain which
avoids aliasing.  So, effectively we approximate the case $H=\Pl$
where $\lambda=\infty$.
Observations of the full-field are made every $J=20$ time-steps.
In Figure \ref{fig:dfni} there is no variance inflation and, whilst
typical ensemble members remain bounded on the time-scales shown,
the error between the ensemble mean and the truth is ${\cal O}(1)$;
indeed comparison with Figure \ref{fig:dfni1} shows that the error
in the mean is in fact {\em worse} than that of an ensemble evolving
without access to data. Using variance inflation removes this problem
and filter accuracy is obtained: see Figure \ref{fig:dfi}.  
The inflation parameter is chosen as $\alpha^2=0.0025$. 

\begin{figure*} 
\center
\includegraphics[width=1\textwidth]{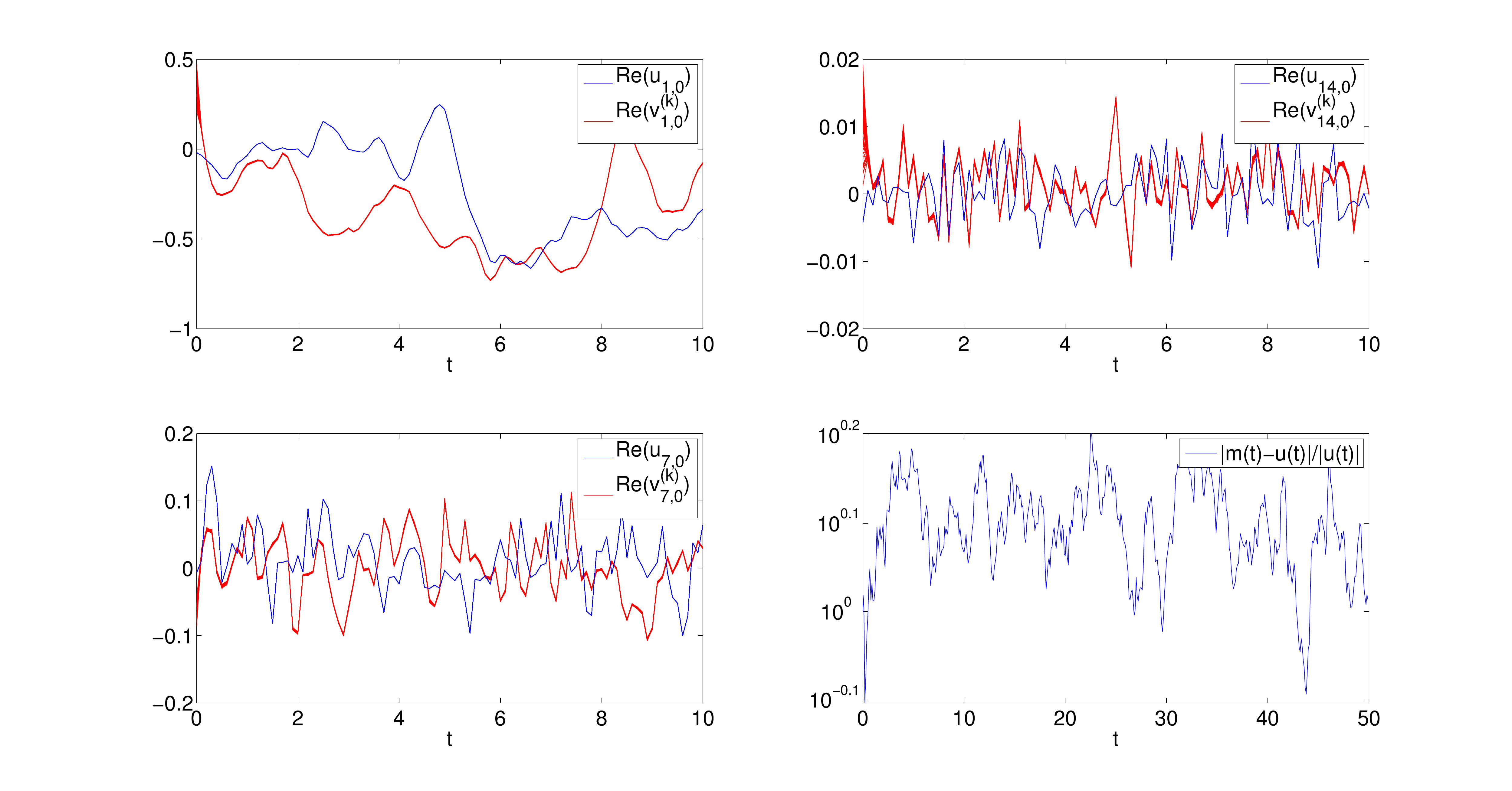}
\caption{Discrete-time observations, without inflation.
Trajectories of various modes  (at observation times) 
of the estimators $v^{(k)}$ 
and the signal $u$ are depicted above 
along with the relative error
in the $L^2$ norm, $|m-u|/|u|$, 
for $H=\Pl$, with $\lambda=\infty$.}
\label{fig:dfni}
\end{figure*}

\begin{figure*} 
\center
\includegraphics[width=1\textwidth]{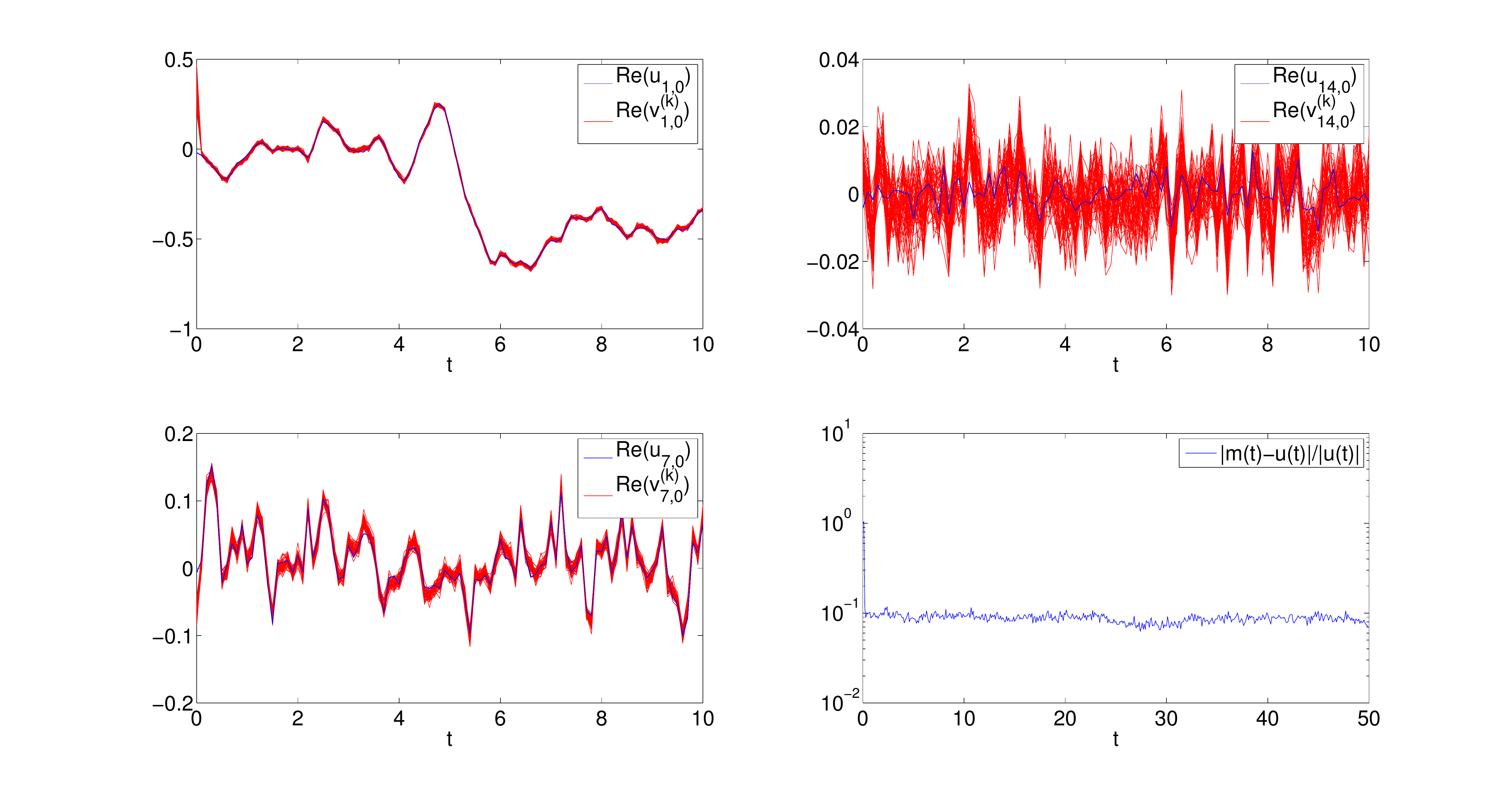}
\caption{Discrete-time observations, with inflation.
Trajectories of various modes (at observation times) 
of the estimators $v^{(k)}$ 
and the signal $u$ are depicted above 
along with the relative error
in the $L^2$ norm, $|m-u|/|u|$, 
for $H=\Pl$, with $\lambda=\infty$.}
\label{fig:dfi}
\end{figure*}

\subsubsection{Partial observations}
\label{sec:discpart}

In this section, the observations are again made every $J=20$ time-steps, 
but we will now consider observing only projections inside and outside a
ring of radius $|k_\lambda|=5$ in Fourier space.  In other words,
we consider two cases $H=\Pl$ and $H=\Ql$ where  
$\lambda=\pi^2|k_\lambda|^2$ and $|k_\lambda|=5$.
This is outside the validity of the theory which considers only full
observations. Inflation is used in both cases.
The inflation parameter is again chosen as $\alpha^2=0.0025$. 
Figure \ref{fig:dpini} shows that when observing all Fourier modes
inside a ring of radius $|k_\lambda|=5$ the filter is accurate
over long time-scales. In contrast, Figure \ref{fig:dpouti} shows
that observing all Fourier modes outside a ring of radius $|k_\lambda|=5$ 
does not provide enough information to induce accurate filtering.


\begin{figure*} 
\center
\includegraphics[width=1\textwidth]{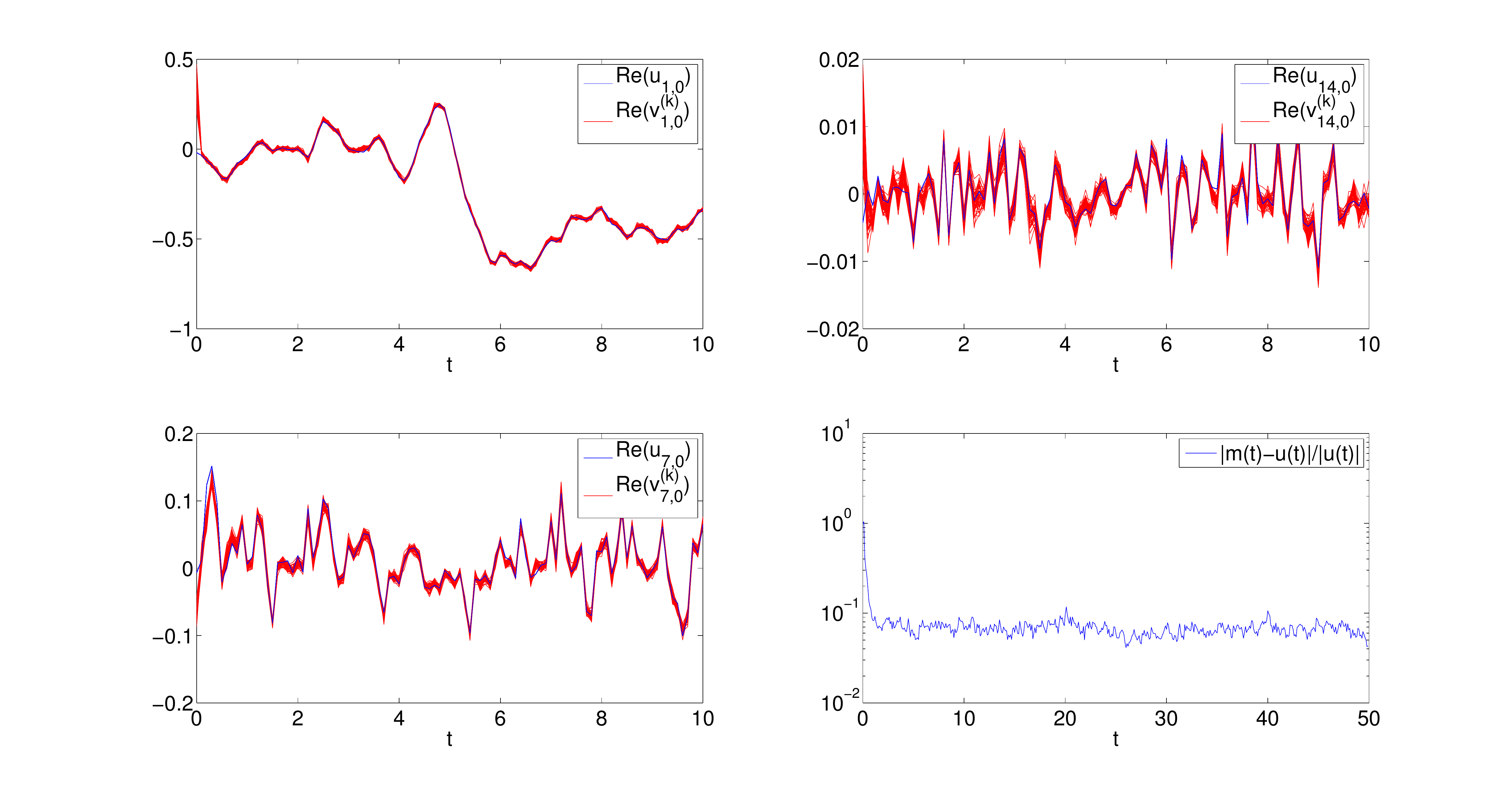}
\caption{Discrete-time observations, with inflation.
Trajectories of various modes  (at observation times) 
of the estimators $v^{(k)}$ 
and the signal $u$ are depicted above 
along with the relative error
in the $L^2$ norm, $|m-u|/|u|$, 
for $H=\Pl$, with $|k_\lambda|=5$, $J=20$, and 
$\gamma=10^{-2}$.}
\label{fig:dpini}
\end{figure*}

\begin{figure*} 
\center
\includegraphics[width=1\textwidth]{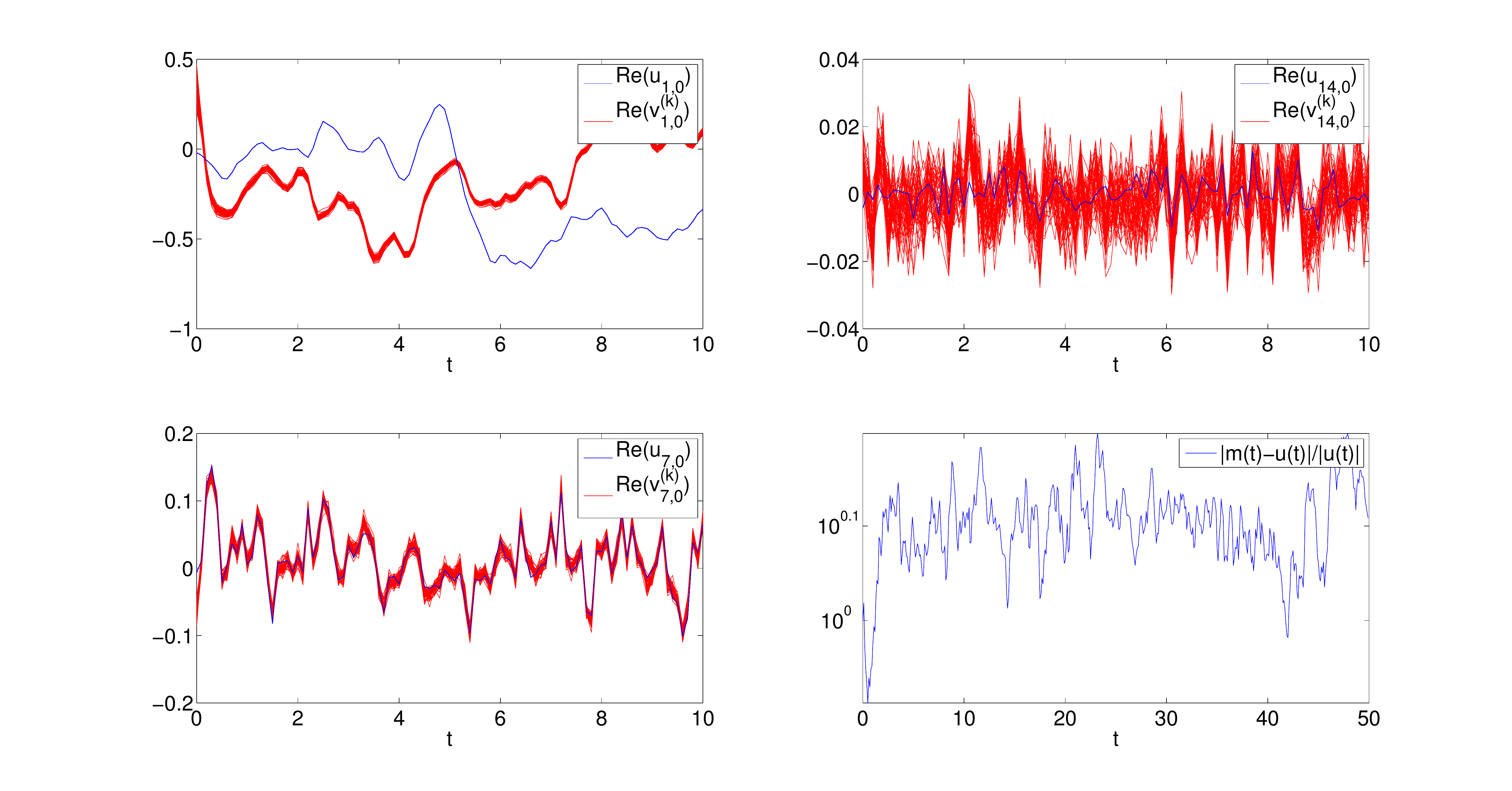}
\caption{Discrete-time observations, with inflation.
Trajectories of various modes  (at observation times) 
of the estimators $v^{(k)}$ 
and the signal $u$ are depicted above 
along with the relative error
in the $L^2$ norm, $|m-u|/|u|$, 
for $H=\Ql$, with $|k_\lambda|=5$, $J=20$, and 
$\gamma=10^{-2}$.}
\label{fig:dpouti}
\end{figure*}

\subsection{Continuous Time}
\label{sec:numcont}

In this section we study the SPDE \eqref{eq:need}, and its relation
to the underlying truth governed by \eqref{e:nse},
by means of numerical experiments. We thereby illustrate and extend
the results of section \ref{sec:enkfc}. 
We invoke a split-step scheme to solve equation \eqref{eq:need}, in which 
for each ensemble member $v^{(k)}$ one  
 step of  numerical integration of the Navier-Stokes
equation \eqref{e:nse} is composed  with one step of numerical integration 
 of the stochastic process
\begin{eqnarray}
d v^{(k)} + 
\frac{1}{\gamma^{2}}C(v)(v^{(k)}(t)-u(t))dt 
= \frac{1}{\gamma}C(v)(dW^{(k)}(t) + dB(t))\\ 
{v^{(k)}}(0)=v^{(k)}_0, \quad m = \frac{1}{K} \sum_{k=1}^K v^{(k)}_0
 \quad C(v) = \frac{1}{K} \sum_{k=1}^K (v^{(k)}_0-m)\otimes (v^{(k)}_0-m)
\label{eq:nse222}
\end{eqnarray}
at each step. The Navier-Stokes equation \ref{e:nse}
itself is solved by the method described in section \ref{sec:numset}.
The stochastic process is also diagonalized in
the Fourier basis \eqref{eq:fb} and then time-approximated by
the Euler-Maruyama scheme \cite{KlPl92}.
We consider the cases $H=\Pl$ with
$\lambda=\infty$ so
that all Fourier modes represented on the grid are observed, as well as
both $H=\Pl$ and $H=\Ql$ with $\lambda<\infty$.

\subsubsection{Full observations}
\label{sec:contfull}

Here we consider observations made at all numerically 
resolved, and hence observable, wavenumbers in the system;
hence $K=32^2$, not including padding in the spectral domain which
avoids aliasing.  So, effectively we approximate the case $H=\Pl$
where $\lambda=\infty$.
Figure \ref{fig:cfni} shows that, without inflation, the ensemble remains
bounded, but the mean is inaccurate, on the time-scales of interest.
In contrast Figure \ref{fig:cfi} demonstrates that inflation leads
to accurate reconstruction of the truth via the ensemble mean.
The inflation parameter is chosen as $\alpha^2=0.00025$. 

\begin{figure*} 
\center
\includegraphics[width=1\textwidth]{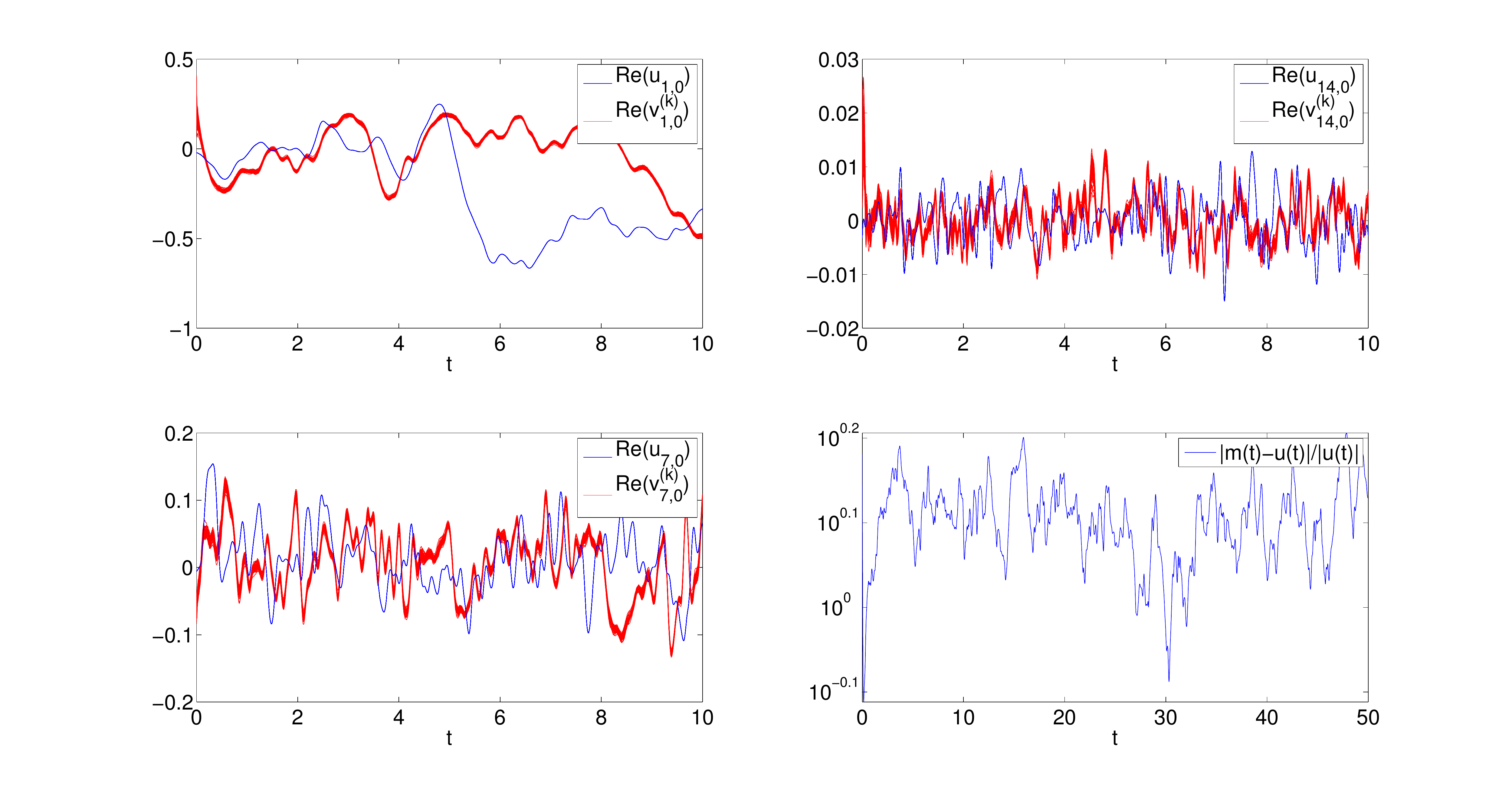}
\caption{Continuous-time observations, without inflation.
Trajectories of various modes of the estimators $v^{(k)}$ 
and the signal $u$ are depicted above 
along with the relative error
in the $L^2$ norm, $|v^{(1)}-u|/|u|$, 
for $H=\Pl$, with $\lambda=\infty$.}
\label{fig:cfni}
\end{figure*}

\begin{figure*} 
\center
\includegraphics[width=1\textwidth]{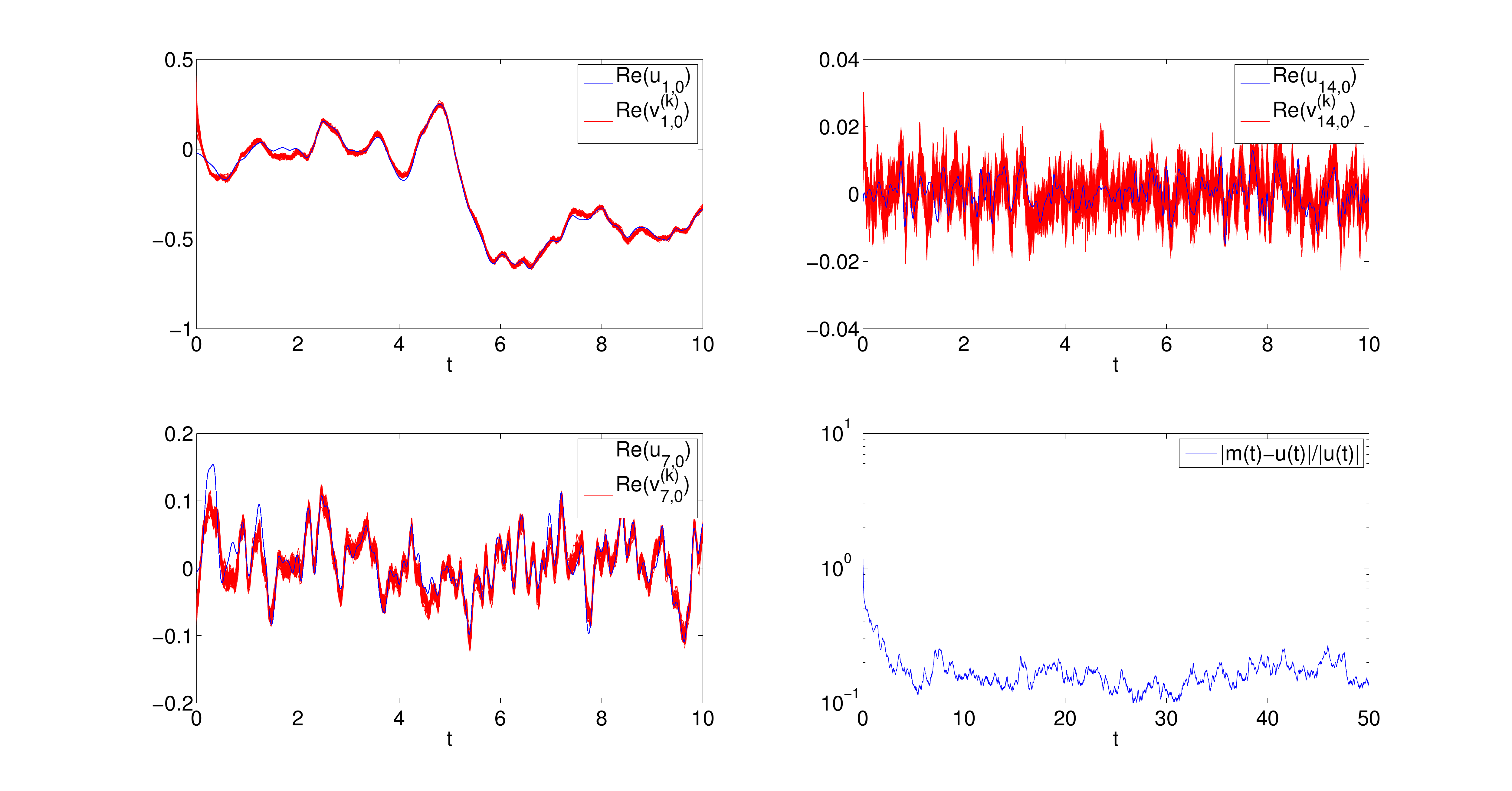}
\caption{Continuous-time observations, with inflation.
Trajectories of various modes of the estimators $v^{(k)}$ 
and the signal $u$ are depicted above 
along with the relative error
in the $L^2$ norm, $|v^{(1)}-u|/|u|$, 
for $H=\Pl$, with $\lambda=\infty$.}
\label{fig:cfi}
\end{figure*}

\subsubsection{Partial observations}
\label{sec:contpart}

Here we consider two cases again, as in section
\ref{sec:discpart} $H=\Pl$ and $H=\Ql$ where  
$\lambda=\pi^2|k_\lambda|^2$, with $|k_\lambda|=5$.
Inflation is used in both cases and the 
inflation parameter is again chosen as $\alpha^2=0.00025$. 
As for discrete time-observations we see that observing inside
a ring in Fourier space leads to filter accuracy (Figure
\ref{fig:cfp}) whilst observing outside yields only filter
boundedness (Figure \ref{fig:cfpc}.)

\begin{figure*} 
\center
\includegraphics[width=1\textwidth]{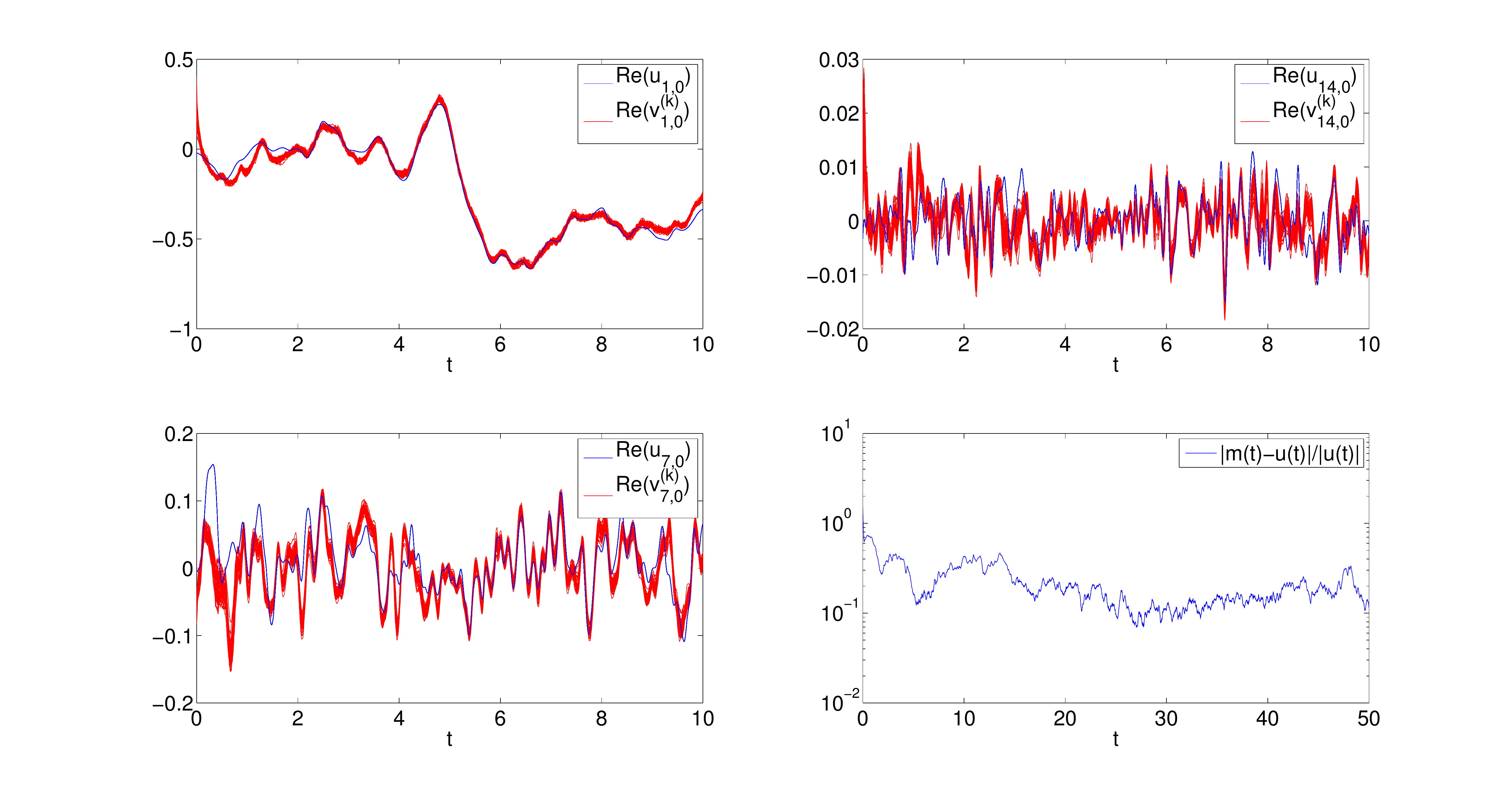}
\caption{Continuous-time observations, with inflation.
Trajectories of various modes of the estimators $v^{(k)}$ 
and the signal $u$ are depicted above 
along with the relative error
in the $L^2$ norm, $|v^{(1)}-u|/|u|$, 
for $H=\Pl$, with $|k_\lambda|=5$.}
\label{fig:cfp}
\end{figure*}

\begin{figure*} 
\center
\includegraphics[width=1\textwidth]{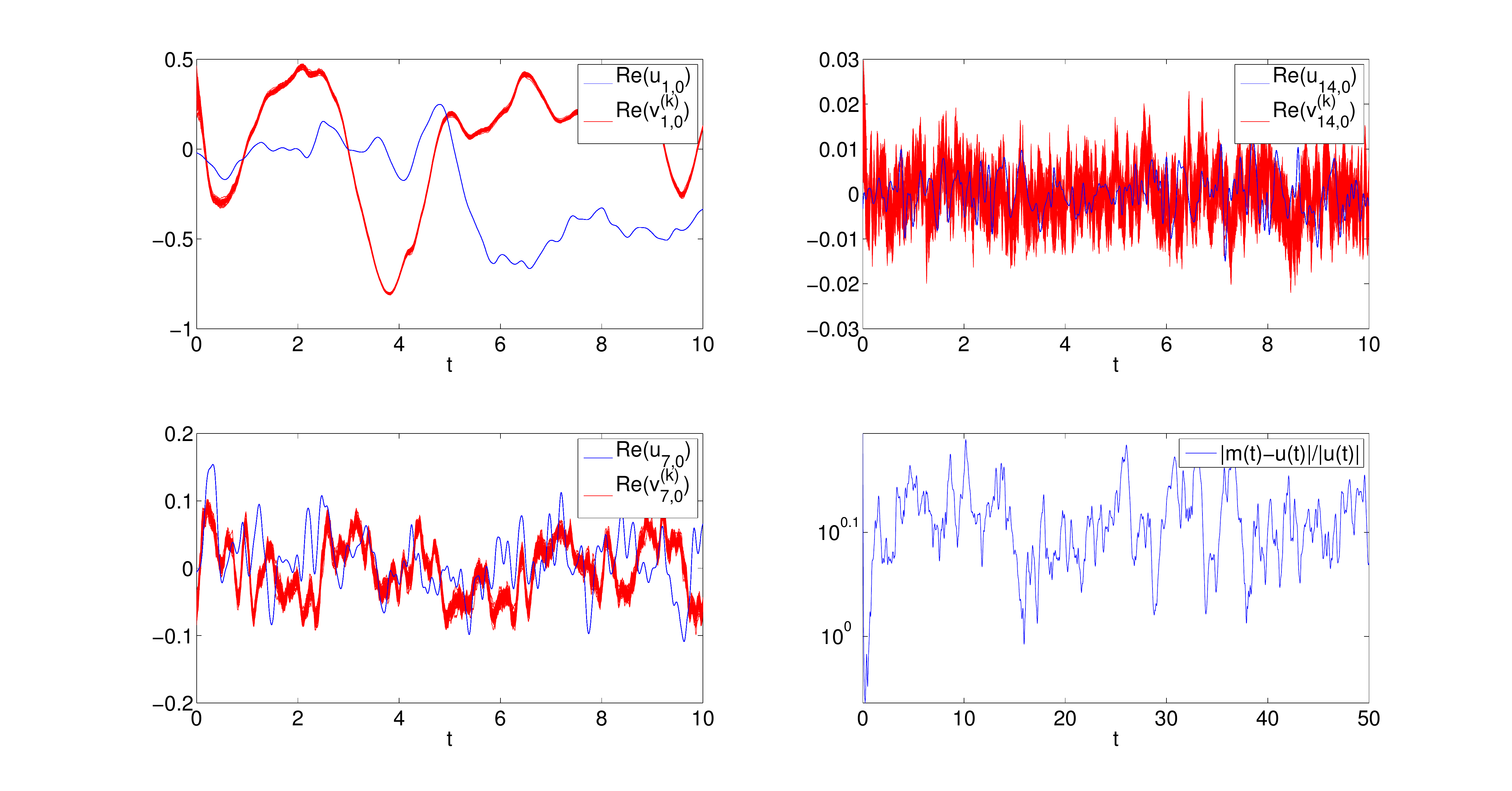}
\caption{Continuous-time observations,  with inflation.
Trajectories of various modes of the estimators $v^{(k)}$ 
and the signal $u$ are depicted above 
along with the relative error
in the $L^2$ norm, $|v^{(1)}-u|/|u|$, 
for $H=\Ql$, with $|k_\lambda|=5$.} 
\label{fig:cfpc}
\end{figure*}

\section{Conclusions}
\label{sec:conc}

We have developed a method for the analysis of the EnKF. Instead of viewing
it as an algorithm designed to accurately approximate the true filtering
distribution, which it cannot do, in general, outside Gaussian scenarios and in the
large ensemble limit, we study it as an algorithm for signal estimation
in the finite (possibly small) ensemble limit. We show well-posedness
of the filter and, when suitable variance inflation is used, mean-square
asymptotic accuracy in the large-time limit. These positive results
about the EnKF are encouraging and serve to underpin its perceived effectiveness
in applications. On the other hand it is important to highlight that
our analysis applies only to fully observed dynamics and interesting
open questions remain concerning the partially observed case.
In this regard it is important to note that the filter divergence
observed in \cite{MH08,gottwald2013} concerns partially observed models.
Thus more analysis remains to be done in this area. The tools introduced
herein may be useful in this regard. A second important direction in which
the analysis could usefully be extended is the class of models to which it
applies. We have studied dissipative quadratic dynamical systems with  
energy conserving nonlinearities. These are of direct relevance in
the atmospheric sciences \cite{kal03} but more general models will be
required for subsurface applications such as those arising in oil
reservoir simulation \cite{orl08}.  {The theoretical results have been confirmed
with numerical simulations of the Navier-Stokes equation on a torus.  These numerical results demonstrate two interesting potential
extensions of our theory: (i) to strengthen well-posedness to obtain 
boundedness of trajectories, at least in mean square; (ii) to extend 
well-posedness and accuracy results to certain partial observation 
scenarios.  Furthermore we highlight the fact that our results have assumed
exact solution of the underlying differential equation model;
understanding how filtering interacts with numerical approximations,
and potentially induces numerical instabilities, is a subject which 
requires further investigation; this issue is highlighted in
\cite{gottwald2013}.}

\vspace{0.1in}
\noindent{\bf Acknowledgements. The work of DK is supported by ONR.
The work of AMS is supported by ERC, EPSRC, ESA and ONR. The authors are
grateful to A.J. Majda for helpful discussions concerning this work.
KJHL is member of KAUST Strategic Research Initiative Center on 
Uncertainty Quantification in Computational Science and Engineering.}

\bibliographystyle{./Martin}
\bibliography{./enkf}

\end{document}